\documentclass{gen-j-l}

\usepackage{tikz}
\usepackage{xcolor}
\usepackage{todonotes}
\usepackage{enumitem}
\usepackage{amscd}
\usepackage{gensymb}
\usepackage{framed}
\usepackage{float}
\usepackage{csquotes}
\usepackage{amssymb}
\usepackage{calc}

\numberwithin{figure}{section}

\usetikzlibrary{decorations.markings,arrows.meta,fit,calc}

\copyrightinfo{2017}{Fabian Parsch}

\newtheorem{theorem}{Theorem}[section]
\newtheorem{lemma}[theorem]{Lemma}
\newtheorem{conjecture}[theorem]{Conjecture}

\theoremstyle{definition}
\newtheorem{definition}[theorem]{Definition}

\theoremstyle{remark}

\newcommand{\ip}[2]{\left\langle#1,#2\right\rangle}
\newcommand{\degrees}{\degree}

\newenvironment{itemize*}
  {\begin{itemize}[topsep=-\parskip+\jot]}
  {\end{itemize}}
\newenvironment{enumerate*}
  {\begin{enumerate}[label=(\alph*),topsep=-\parskip+\jot]}
  {\end{enumerate}}
\newenvironment{enumerate**}
  {\begin{enumerate}[label=(\roman*),topsep=-\parskip+\jot]}
  {\end{enumerate}}

\newenvironment{enumerate***}
  {\begin{enumerate}[label=(\alph*'),topsep=-\parskip+\jot]}
  {\end{enumerate}}
  
\begin{document}

\title{Geodesic nets with three boundary vertices}

\author{Fabian Parsch}
\address{Fabian Parsch, Department of Mathematics, University of Toronto, 40 St. George Street, Toronto, ON M5S 2E4, Canada}
\email{fparsch@math.toronto.edu}

\begin{abstract}
	We prove that a geodesic net with three boundary (= unbalanced) vertices on a non-positively curved plane has at most one balanced vertex. We do not assume any a priori bound for the degrees of unbalanced vertices.

	The result seems to be new even in the Euclidean case.

	We demonstrate by examples that the result is not true for metrics of positive curvature on the plane, and that there are no immediate generalizations of this result for geodesic nets with four unbalanced vertices.
\end{abstract}

\maketitle

\section{Introduction}
\label{sect:intro}

Geodesic nets are graphs embedded in a manifold such that each edge is a geodesic segment. Furthermore, we require that at each \emph{balanced} vertex, the unit tangent vectors of the edges cancel, i.e. their sum is zero. All other vertices are called \emph{unbalanced}. We consider such nets where each edge has weight one.

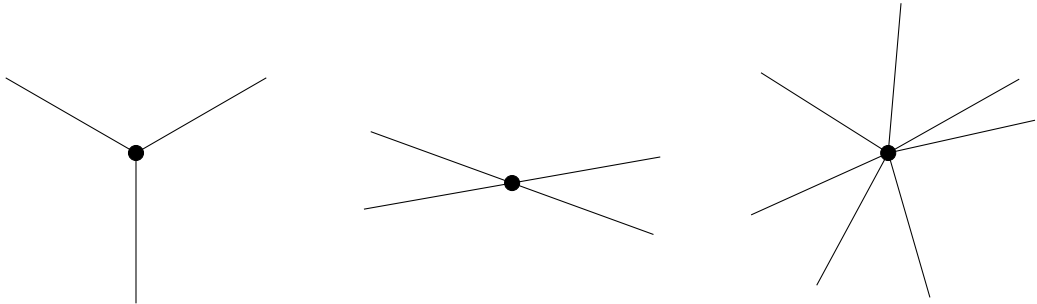
\begin{figure}
	\centering
	\begin{tikzpicture}[scale=2]
		\draw (-1,0) -- +(30:1);
		\draw (-1,0) -- +(150:1);
		\draw (-1,0) -- +(270:1);
		\draw [fill=black] (-1,0) circle [radius=0.05];
		
		\draw (1.5,-0.2) -- +(10:1);
		\draw (1.5,-0.2) -- +(-20:1);
		\draw (1.5,-0.2) -- +(160:1);
		\draw (1.5,-0.2) -- +(190:1);
		\draw [fill=black] (1.5,-0.2) circle [radius=0.05];
		
		\draw (4,0) -- (4.085159600262456,0.9963673230707331);
		\draw (4,0) -- (3.154511096969029,0.5339929913879817);
		\draw (4,0) -- (3.088735454976287,-0.4118214770780241);
		\draw (4,0) -- (4.975942150302623,0.2180296293229263);
		\draw (4,0) -- (3.5248510852651163,-0.8799053976571926);
		\draw (4,0) -- (4.2773364928491535,-0.9607728502274256);
		\draw (4,0) -- (4.870934923071531,0.49139837176611095);
		\draw [fill=black] (4,0) circle [radius=0.05];
	\end{tikzpicture}
	\caption{Examples for balanced vertices of degree $3$, $4$ and $7$. }
	\label{fig:balancedvertices}
\end{figure}

In \cite{iejg347253}, Memarian studies a question about such planar geodesic nets, i.e. geodesic nets in flat $\mathbb{R}^2$, (which he calls \emph{critical graphs}) by Gromov related to his work in \cite{MR2563769}: Given a number of unbalanced vertices, each of degree one, what is the maximal number of balanced vertices in a geodesic net \enquote{spanned} by the unbalanced vertices. In the following we will turn towards a more general question than the one studied by Memarian (for a comparison, see below).

The study of geodesic nets bears some relation to minimal networks and the Steiner Problem as studied by Ivanov and Tuzhilin (for a survey, see \cite{MR3616371}). In fact, Corollary 1.1 in Chapter 3 of \cite{MR1271779} states that geodesic nets are local minima (with respect to the length of all edges) of immersed parametric networks. The emphasis has to be put on \emph{immersed}, though: In fact, in their variational study of the Steiner Problem, Ivanov and Tuzhilin, later allow for degeneration. That means that when embedding the graph, several vertices can be mapped to the same image, abandoning the immersive property. They show that in most situations, that variational problem arrives at graphs with vertices of degree $3$ that then minimize the total length of all edges given fixed boundary vertices. Examples like the one provided in Figure \ref{fig:fournet}, however, show that if we want to maximize the number of vertices in a geodesic net instead of minimizing the length of the edges, allowing vertices of degree higher than $3$ allows for many more balanced vertices.

\subsection{Geodesic nets}

We are studying geodesic nets which are defined as follows.

\begin{definition}\label{def:geodesicnet}
	Let $G=(V,E)$ be a connected, finite graph embedded in a surface $\Sigma$ such that each edge is a geodesic.
	
	We call $G$ a \emph{geodesic net} by denoting as $B\subset V$ the set of \emph{balanced vertices} consisting of all vertices $v$ for which the following is true: The vertex has at least degree $3$ and the sum of all unit vectors in the tangent space $T_v\Sigma$ tangent to the edges incident to $v$ and directed from $v$ is zero.

	Accordingly, we call $U=V\setminus B$ the set of \emph{unbalanced vertices} and say that $G=(B,U,E)$ is a geodesic net.
\end{definition}
	
\begin{figure}[b]
	\centering
	\begin{tikzpicture}[scale=0.8]
		\draw [line width=0.5pt] (0,0)-- (20:2.7);
		\draw [line width=0.5pt] (0,0)-- (140:1);
		\draw [line width=0.5pt] (0,0)-- (260:1.8);
	\end{tikzpicture}
	\caption{A geodesic net with $3$ unbalanced vertices and $1$ balanced vertex (the \textit{Fermat Point}). We will show that this is in fact the maximal number of balanced vertices when only given $3$ unbalanced vertices on the plane with a metric of nonpositive curvature.}
	\label{fig:threenet}
\end{figure}
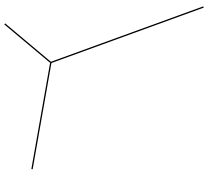
	
\begin{figure}[b]
	\centering
	\begin{tikzpicture}[scale=0.7]
	\clip (-5.5,1.5) rectangle (5.5,12);
	\node at (0,0) (Q) {};
	\node at (150:3) (X) {};
	\node at (125:3) (Y1) {};
	\node at (90:3) (Y2) {};
	\node at (55:3) (Y3) {};
	\node at (30:3) (Z) {};
	
	\node at (0,15) (P) {};
	\path (P) -- ++(210:6) node (A) {};
	\path (P) -- ++(255:6) node (B1) {};
	\path (P) -- ++(270:6) node (B2) {};
	\path (P) -- ++(285:6) node (B3) {};
	\path (P) -- ++(330:6) node (C) {};
	
	\node at (-.84,4.48) (L) {};
	\node at (.84,4.48) (N) {};

	\draw (A.center) -- (B1.center) -- (C.center);
	\draw (A.center) -- (B2.center) -- (C.center);
	\draw (A.center) -- (B3.center) -- (C.center);
	
	\draw (X.center) -- (Y1.center) -- (Z.center);
	\draw (X.center) -- (Y2.center) -- (Z.center);
	\draw (X.center) -- (Y3.center) -- (Z.center);
	
	\draw (B1.center) -- (X.center);
	\draw (B3.center) -- (Z.center);
	
	\draw (Y1.center) -- (A.center);
	\draw (Y3.center) -- (C.center);
	
	\draw (B2.center) -- (Y2.center);
	
	\draw (A.center) -- (L.center);
	\draw (X.center) -- (L.center);
	\draw (Z.center) -- (N.center);
	\draw (C.center) -- (N.center);
	\draw (L.center) -- (N.center);
	
	\draw (A.center) -- (Z.center);
	\draw (C.center) -- (X.center);
\end{tikzpicture}
	\caption{An example of a geodesic net in the plane with four unbalanced vertices and $27$ balanced vertices of which eight have degree $3$, eighteen have degree $4$ and one has degree $6$. For a detailed construction, see Figure \ref{fig:construct4}.}
	\label{fig:fournet}
\end{figure}
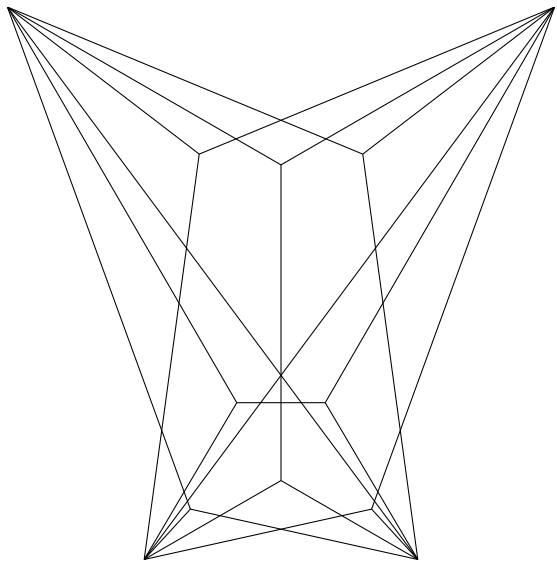

One might be inclined to use \emph{interior} and \emph{boundary} vertex for balanced and unbalanced vertices respectively. Note, however, that the above definition does not require that all unbalanced vertices lie on the boundary of the geodesic net (i.e. on the boundary of the union of its vertices, edges and faces). An unbalanced vertex can lie inside the convex hull of the net. In light of lemma \ref{thm:convexhull}, for the case of three unbalanced vertices on the plane, the expressions \emph{boundary vertex} and \emph{unbalanced vertex} are interchangeable, though.

Two of the most intuitive examples of geodesic nets on the flat plane with \textit{three} unbalanced vertices are described as follows:
\begin{itemize*}
	\item A triangle with its three vertices and three edges is a geodesic net with three unbalanced and no balanced vertices.
	\item On the other hand, we can consider three points arranged in a triangle with all interior angles less than $120\degrees$. We can position a point inside the triangle so that, if connected to the three corners, this point is balanced (see figure \ref{fig:threenet}). That such a \textit{Fermat point} exists is a result of classic Euclidean geometry.
\end{itemize*}

\subsection{Main question}

The question which is considered in the Main Theorem \ref{thm:mainshort} is described by:

\textit{Given the number of unbalanced vertices of a geodesic net, what is the maximal number of balanced vertices that the net can have?}

This question is quite general in the following sense: We are assuming no a priori information on the distance/relative position of the unbalanced vertices. That is the bound should only depend on the number of unbalanced vertices. We also allow the unbalanced vertices to have arbitrary degree.

We will show that, if the surface is $\mathbb{R}^2$ with a metric of nonpositive curvature (including flat and hyperbolic space as two special cases), then the configuration with the Fermat Point in Figure \ref{fig:threenet} is in fact maximal. Here is our main theorem:

\begin{theorem}[Main Theorem]\label{thm:mainshort}
	Each geodesic net with 3 unbalanced vertices (of arbitrary degree) on the plane endowed with a Riemannian metric of non-positive curvature has at most one balanced vertex.
\end{theorem}

In fact, the theorem is new even in the case when dealing with geodesic nets in the Euclidean plane (and the proof is almost as difficult as in the general case). On the other hand the result is false for metric of positive curvature on the plane. In section \ref{sect:positive}, we exhibit and example of a geodesic net with 3 unbalanced and 3 balanced vertices on the round hemisphere (of course this Riemannian metric could be extended to a positively curved metric on the whole plane).

The result is somewhat surprising in the context of Figure \ref{fig:fournet} which is showing that with four unbalanced vertices in the plane, it is possible to have at least $27$ balanced vertices. Note that we do not claim that the example in this figure is maximal. We will give a detailed description of this example in section \ref{sect:construct4}.

This example shows that for larger numbers of unbalanced vertices, the number of balanced vertices increases significantly, adding credibility to the following conjecture that we learned from Alexander Nabutovsky:

\begin{conjecture}\label{conj:limitless}
	There exists $N_0$ and a sequence of geodesic nets in the plane with $N_0$ unbalanced vertices and an arbitrarily large number of balanced vertices.
\end{conjecture}

On the other hand, Gromov's question can be restated in the following equivalent form (It is not difficult to prove that the two formulations are equivalent, but we are not going to present the proof in this paper): One can always estimate the number of balanced vertices of a geodesic net in terms of the number of unbalanced vertices \textit{and} the total imbalance, described as follows.

\begin{definition}
	For a vertex $v$, denote by $\textrm{imb}(v)$ the \emph{imbalance} defined as the norm of the sum of unit vectors tangent to the incident edges, i.e. $v$ is balanced iff $\textrm{imb}(v)=0$.
\end{definition}

\begin{conjecture}\label{conj:imbalance}
	There is a function $g:\mathbb{N}\times\mathbb{R}\to\mathbb{N}$ such that for all geodesic nets $G=(B,U,E)$ in the plane with at most $n$ unbalanced vertices (i.e. with $|U|\leq n$) and total imbalance $c$ or less, i.e. with
	\begin{align*}
		\sum_{v\in U}\textrm{imb}(v)\leq c
	\end{align*}
	we have $|B|\leq g(n,c)$.
\end{conjecture}

Note that this conjecture is not true for geodesic nets on arbitrary surfaces, even flat surfaces. On the flat torus, we can take an arbitrary number of closed geodesics whose union is connected. This produces geodesic nets with zero imbalance where each point of intersection is a balanced vertex. More generally, on any surface with periodic geodesics, one can construct geodesic nets with extra balanced vertices by combining geodesic nets with periodic geodesics that intersect it. As long as there is an infinite number of such intersecting closed geodesics, this construction gives an arbitrary number of balanced vertices without adding to the imbalance or the number of unbalanced vertices.

\subsection{Regarding the degree of balanced vertices}

Note that we require balanced vertices to have degree $3$ or more. In fact allowing degree $2$ balanced vertices would render the question meaningless: Obviously, one could add an arbitrary number of degree $2$ balanced vertices to the edges of any geodesic net.

\subsection{Regarding the degree of unbalanced vertices}

Note that we do not put any bound on the degree of the unbalanced vertices. In other words, we allow a single unbalanced vertex to be adjacent to an arbitrary number of balanced vertices.

\subsection{Previous results}

In \cite{iejg347253}, Memarian considered the question on the Euclidean plane with one restriction: each unbalanced vertex has degree $1$. He studied two special cases: If all balanced vertices have degree $3$, he finds a sharp upper bound for what he calls a \emph{3-boundary regular critical graphs}. The bound is achieved by what resembles a tiling of the plane by hexagons. If all balanced vertices have degree $4$, he notes the (presumably not sharp) upper bound given by the maximal number of intersections between straight line segments.

As Memarian points out, though, these two cases depend highly on the geometric restrictions that can be presumed for balanced vertices of degree $3$ or degree $4$: Indeed, for degree $3$, the edges must be arranged equiangularly around the vertex with angles of $120\degrees$ between them. For degree $4$, the vertex and its incident edges are given as the intersection of two straight lines. The example in figure \ref{fig:balancedvertices} shows that for higher degrees, on the other hand, balanced vertices can be quite irregular.

This leaves the problem open for the case that we do not put a limit on the degree of the balanced vertices, and for the case of nonzero curvature. In fact, even the case of a planar geodesic net that has a mix of nothing but degree $3$ and degree $4$ balanced vertices is left open. Furthermore, as stated before, we will not require the unbalanced vertices to have degree $1$ but instead allow arbitrary degree.

As mentioned, we will provide a counterexample that shows that the Theorem fails on the plane with positive curvature. Note that if we considered the round sphere, intersecting great circles will construct an arbitrary number of balanced vertices (see above), so the question about the number of vertices becomes trivial. Instead, in this case, the existence of geodesic nets of particular shapes is a relevant question. There is only a small number of papers in this area. In \cite{Heppes1964}, Heppes classifies so-called \enquote{isogonal nets} on the round 2-sphere which -- using the language of the present paper -- are geodesic nets that have no unbalanced vertices and for which at each balanced vertex the incident edges are equiangularly distributed. In \cite{MR1343696}, Hass and Morgan study the existence of geodesic nets dividing a 2-sphere with a positively curved metric into a specific number of regions. A corollary in their work is that in all convex metrics in a neighbourhood of the round metric on the 2-sphere, there exists a geodesic net homeomorphic to a \enquote{$\theta$-graph} (a graph consisting of three edges meeting in two vertices). This is remarkable since unlike other results it doesn't prove existence of particular geodesic nets in a specific given metric, but in an open set of metrics. They finish with a conjecture that any smooth 2-sphere can be divided into $n$ regions by a geodesic net using nothing but degree $3$ and $4$ vertices. In \cite{MR1618690}, Heppes confirms this conjecture for the case of the round sphere. In \cite{MR2836083} and \cite{MR2171308} on the other hand, Nabutovsky and Rotman study upper bounds on the length of geodesic nets of particular shape.

\section{Structure of the proof}

From now on, unless specified otherwise, any graph discussed will be a geodesic net $G=(B,U,E)$ as given by Definition \ref{def:geodesicnet}.

We will prove the following reformulation of the main theorem.

\begin{theorem}\label{thm:main}
	Define $f:\mathbb{N}\to\mathbb{N}\cup\{\infty\}$ as follows: $f(n)$ is the smallest number such that $|B|\leq f(|U|)$ is true for all geodesic nets on $\mathbb{R}^2$ with a metric of  nonpositive curvature. Then
	\begin{enumerate*}
		\item $f(0)=f(1)=f(2)=0$
		\item $f(3)=1$
	\end{enumerate*}
\end{theorem}

Note that the bound for $n=0,1$ is obvious and the bound for $n=2$ is nearly trivial as will be seen in Lemma \ref{thm:convexhull}. The case for $n=3$ is the actually interesting result that we will prove.

To prove the theorem, we proceed as follows: We will first study the properties of a single balanced vertex in Section \ref{sect:local}. In particular, we will prove restrictions on the angles between the edges. We will use these properties in Section \ref{sect:global} when we turn to global properties on the plane and study how the angles between edges not incident to the same vertex are related by introducing the \emph{turn angle} along a path. We will then prove the result regarding three vertices on the plane in Section \ref{sect:proof3}. The results of Sections \ref{sect:global} and \ref{sect:proof3} will then be generalized to the case of nonpositive curvature in Section \ref{sect:negative}.

This theorem obviously asks for an extension to larger $n$ or positive curvature. We will construct an example for a gedesic net with $n=4$ in Section \ref{sect:construct4}. Furthermore, in Section \ref{sect:positive}, we will construct a geodesic net on a surface of positive curvature with just three unbalanced vertices but more than one balanced vertex. This example shows that the Main Theorem can't be true for metrics of positive curvature on $\mathbb{R}^2$, even when no closed geodesics exist.

\section{Local Properties}
\label{sect:local}

In the following, we prove helpful lemmas that describe local properties in the sense that they \enquote{zoom in} on a single vertex without considering properties of any other vertices of the geodesic net. It is important to point out that these local properties apply to the vertices of a geodesic net on any surface, no matter the curvature since we only consider the tangent space at a single vertex.

\subsection{General Local Properties}

\begin{definition}
	Generally, if we consider several edges incident to the same vertex, we enumerate them in counterclockwise order, e.g. when we say \enquote{$b$ directly follows $a$}.
\end{definition}

Note the following two facts:

\begin{lemma}\label{thm:straightline}
	If we draw a geodesic through any balanced vertex, there must always be an edge on each side of that line.
\end{lemma}

\begin{proof}
	Recall that a balanced vertex has at least degree $3$ which means even after discounting for the possibility that two edges lie on the geodesic, there must be another edge lying on one side of it. To balance the unit vector parallel to that edge, one needs an edge on the other side of that geodesic, too.
\end{proof}

\begin{lemma}\label{thm:lessthan180}
	The angle between an edge at a balanced vertex and its immediately following edge must be less than $180\degrees$.
\end{lemma}

\begin{proof}
	Otherwise there would be a geodesic through the vertex such that there are no edges on one side of it, contradicting Lemma \ref{thm:straightline}.
\end{proof}

\subsection{Combined Angles}

\begin{definition}
	Consider a balanced vertex $v$ and three incident edges $a,b,c$ following in that counterclockwise order without any edges in between. Then the \emph{combined angle of $b$ at $v$} is the total angle from $a$ to $b$ to $c$ (see Figure \ref{fig:combinedangle}).
\end{definition}

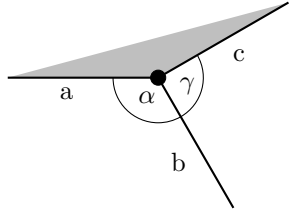
\begin{figure}[t]
\centering
\begin{tikzpicture}[scale=2]
	\node at (180:1) (aend) {};
	\node at (300:1) (bend) {};
	\node at (30:1) (cend) {};
	\path [fill=lightgray] (0,0) -- (cend.center) -- (aend.center) -- cycle;
	\draw [thick] (0,0) -- node [below left] {a} (aend.center);
	\draw [thick] (0,0) -- node [below left] {b} (bend.center);
	\draw [thick] (0,0) -- node [below right] {c} (cend.center);
	\draw [fill=black] (0,0) circle [radius=0.05];
	\draw (180:0.3) arc [radius=0.3,start angle = 180,end angle = 300];
	\draw (300:0.3) arc [radius=0.3,start angle = 300,end angle = 390];
	\node at (240:0.15) {$\alpha$};
	\node at (345:0.2) {$\gamma$};
\end{tikzpicture}
\caption{The combined angle of $b$ at this vertex is $\alpha+\gamma$. Note that all other edges incident to the vertex are in the grey area.}
\label{fig:combinedangle}
\end{figure}

\begin{lemma}[General Combined Angle Lemma]\label{thm:combinedangle}
	The combined angle of any edge at a balanced vertex is...
	\begin{enumerate*}
		\item ... equal to $240\degrees$ if the vertex has degree $3$.
		\item ... equal to $180\degrees$ if the vertex has degree $4$.
		\item ... strictly smaller than $180\degrees+2\arcsin 1/(n-1)$ if the vertex has degree $n\geq 5$.
	\end{enumerate*}
	In particular the combined angle is always less or equal than $240\degrees$ and strictly so if the vertex has degree larger than three.
\end{lemma}

\begin{proof}
	\begin{enumerate*}
		\item is obvious.
		\item is obvious.
		\item consider a vertex of degree $n\geq 5$. Assume that the combined angle at an edge $b$ is $180\degrees+2\arcsin 1/(n-1)$ or more. Call the two edges realizing that angle $a$ and $c$. Take $v$ to be the unit vector that bisects the smaller of the two angles between $a$ and $c$. Let $\{e_i\}$ be the edges other than $a$, $b$, $c$. There are at least two such edges since the degree is $5$ or more. Note that the $e_i$ lie on the side of the angle formed by $a$ and $c$ that does not contain $b$. By a slight abuse of notation we use the same names for the edges and the corresponding unit vectors (Compare figure \ref{fig:combinedangleproof}). Since the combined angle is $180\degrees+2\arcsin 1/(n-1)$ or more, basic trigonometry yields:
			\begin{align*}
				\ip{e_i}{v}>\ip{a}{v}\geq\sin((180\degrees+2\arcsin1/(n-1)-180\degrees)/2)=1/(n-1)
			\end{align*}
		
			We deduce:
			\begin{align*}
				0=&\ip{b+a+c+\sum e_i}{v}\\
				=&\underbrace{\ip{b}{v}}_{\geq-1} + \underbrace{\ip{a}{v}}_{\geq1/(n-1)}+ \underbrace{\ip{c}{v}}_{\geq1/(n-1)} + \underbrace{\sum\ip{e_i}{v}}_{>(n-3)/(n-1)}\\
				>&-1+1/(n-1)+1/(n-1)+(n-3)/(n-1)=0
			\end{align*}
			This is a contradiction.
	\end{enumerate*}
\end{proof}

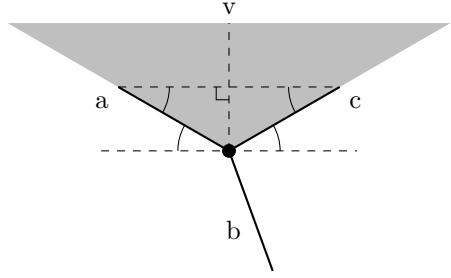
\begin{figure}
\centering
\begin{tikzpicture}[scale=1.7]
	\path [fill=lightgray] (150:2) -- (0,0) -- (30:2) -- cycle;
	\draw [dashed] (-1,0) -- (1,0);
	\draw [fill=black] (0,0) circle [radius=0.05];
	\draw [thick] (0,0) -- node [below left] {b} (-70:1);
	\draw [thick] (0,0) -- (150:1) node [below left] {a};
	\draw [thick] (0,0) -- (30:1) node [below right] {c};
	\draw [dashed] (0,0) -- (90:1) node [above] {v};
	\draw [dashed] (150:1) -- (30:1);
	\draw (180:0.4) arc [radius=0.4,start angle = 180,end angle = 150];
	\draw (0:0.4) arc [radius=0.4,start angle = 0,end angle = 30];
	\draw (150:0.6) arc [radius=0.4,start angle = -30,end angle = 0];
	\draw (30:0.6) arc [radius=0.4,start angle = 210,end angle = 180];
	\draw (-0.1,0.5) -- (-0.1,0.4) -- (0,0.4);
\end{tikzpicture}
\caption{Proof of Lemma \ref{thm:combinedangle}. Note that if we denote the combined angle of $b$ by $\delta$, then the marked angles are equal to $(\delta-180\degrees)/2$ which leads to the stated formula for the projection of the vectors onto $v$. All other edges/vectors must lie in the grey area.}
\label{fig:combinedangleproof}
\end{figure}

In the next lemma we will show that for vertices of degree $n\neq 4$, only at a vertex of odd degree can we have a combined angle of $180\degrees$ or more, and even then particular restrictions to the angles apply.

\begin{lemma}[Special Combined Angle Lemma]\label{thm:specialcombinedangle}
	Let $a,b,c$ be three directly following edges of a balanced vertex of degree $n\geq 5$ with a combined angle of $b$ that is $180\degrees$ or more. Then the vertex must have odd degree.
	
	Furthermore, denote by $\alpha$ the angle between $a$ and $b$ and by $\gamma$ the angle between $b$ and $c$ (i.e. $\alpha+\gamma$ is the combined angle of $b$) then:
	\begin{enumerate*}
		\item $60\degrees<\alpha< 120\degrees$
		\item $60\degrees<\gamma< 120\degrees$
	\end{enumerate*}
\end{lemma}

Note that the result of $n$ being odd will not be used when we apply this lemma later. But we get it as a \enquote{gratuitous result} which is worth noting.

\begin{proof}
	We will prove the inequalities regarding $\alpha$. The case for $\gamma$ then follows by reflection and relabeling.
	
	We arrange and label the edges as follows: Choose a coordinate system in the tangent space at the vertex such that the unit vector corresponding to $a$ lies on the negative $x$-axis. Since the combined angle at $b$ is $180\degrees$ or more, this implies that only the unit vector of $b$ could lie in the lower half plane. And in fact by Lemma \ref{thm:lessthan180}, it must lie in the lower half plane (see figure \ref{fig:combinedangle}).
	
	From now on, we denote by $b$ the unit vector corresponding to that edge and by $e_1:=a,e_2:=c,e_3,\dots,e_{j}$ ($j:=n-1\geq4$) the unit vectors corresponding to all other edges, including $a$ and $c$ as the first two. By the above setup, all $e_i$ have a nonnegative $y$-component and only $e_1=a$, $e_2=c$ can have a zero $y$-component.
	
	Note that if we define $s:=\sum e_i$, then $s=-b$ by the balancing condition and therefore $s$ is a unit vector. This means the lemma follows if we prove the following two claims:
	\begin{enumerate***}
		\item $j=n-1$ must be even.
		\item If $j=n-1$ is even, then the angle between $s$ and the positive $y$-axis lies within $\pm 30\degrees$ (Then, the angle $\alpha$ is between $60\degrees$ and $120\degrees$).
	\end{enumerate***}
  	We will now prove (a') and (b').
  	
  	\textbf{Claim:} None of the $e_i$ can lie on the positive $y$-axis.\\
  	If that were not the case, note that there is at least four $e_i$ in total, one of them lying on the $y$-axis. Even if another two lie on the $x$-axis (and therefore have zero $y$-component), there is a fourth one that has a positive $y$-component. Therefore the sum $s=\sum e_i$ would have a $y$-component of more than $1$ which contradicts the fact that $s$ must be a unit vector.
  	
  	Due to the above claim, we can group the $e_i$ according to the following rules:
  	\begin{itemize*}
  		\item In one group, we have all $e_i$ pointing to the left (negative x-coordinate). Call it the \emph{left group}.
  		\item Accordingly, the other vectors $e_i$ form the \emph{right group}.
  	\end{itemize*}
  	Note that by our setup, $a$ is in the left group. Also note that by the General Combined Angle Lemma \ref{thm:combinedangle} we have $\alpha+\gamma\leq 240\degrees$ and therefore $c$ must be in the right group.
  	
  	We denote the number of left and right vectors by $L$ and $R$ respectively. Note that $R+L=j$.
  	
  	In the following, we will use the notation $\langle x,y\rangle$ to mean a vector in $\mathbb{R}^2$ with coordinates $x$ and $y$.
  	
  	Let $e:\{-1,1\}\times [0\degrees,90\degrees)\to \mathbb{R}^2$ be given by
  	\begin{align*}
  		e(C,\theta):=\langle C\cos\theta,\sin\theta\rangle
  	\end{align*}
  	We can rewrite each of the right vectors as $e_i=e(1,\theta_i)$ where $\theta_i$ is the angle between the positive $x$-axis and $e_i$, and each of the left vectors as $e_i=e(-1,\theta_i)$ where $\theta_i$ is the angle between the negative $x$-axis and $e_i$. Note that in either case, $0\degrees\leq\theta_i<90\degrees$. Using this notation, we can write
  	\begin{align*}
  		s=e_1+e_2+\dots+e_j=e(C_1,\theta_1)+e(C_2,\theta_2)+\dots+e(C_j,\theta_j)
  	\end{align*}
  	In the proof below, we will start at $\langle R-L,0\rangle$ and arrive at the actual vector sum $s$ by iteration as follows:
  	\begin{align*}
  		s_0&=e(C_1,0\degrees)+e(C_2,0\degrees)+e(C_3,0\degrees)+\dots+e(C_j,0\degrees)=\langle R-L,0\rangle\\
  		s_1&=e(C_1,\theta_1)+e(C_2,0\degrees)+e(C_3,0\degrees)+\dots+e(C_j,0\degrees)\\
  		s_2&=e(C_1,\theta_1)+e(C_2,\theta_2)+e(C_3,0\degrees)+\dots+e(C_j,0\degrees)\\[-2mm]
  		&\vdots\\[-2mm]
  		s_j&=e(C_1,\theta_1)+e(C_2,\theta_2)+e(C_3,\theta_3)+\dots+e(C_j,\theta_j)=s
  	\end{align*}
  	We will use this iterative process to prove that $s$ will lie in a \enquote{staircase region}, defined as follows (see also figure \ref{fig:arcfact}):
  	
  	\textbf{Definition:} We define the \emph{leftwards staircase of unit circles} starting at $\langle R-L,0\rangle$ as the union of the counterclockwise quarter arcs of unit circles starting at $\langle R-L-\ell,\ell\rangle$ and ending at $\langle R-L-(\ell+1),\ell+1\rangle$ for $\ell=0,1,2,\dots$. We define the \emph{rightwards staircase of unit circles} starting at $\langle R-L,0\rangle$ as the union of the clockwise quarter arcs of unit circles starting at $\langle R-L+\ell,\ell\rangle$ and ending at $\langle R-L+\ell+1,\ell+1\rangle$ for $\ell=0,1,2,\dots$. The region between these two staircases, including the boundary, is called the \emph{staircase region}.
  	
  	\clearpage
  	Returning to our definition of the map $e(C,\theta)$ and of the sequence $s_k$ above, we can describe the step from $s_k$ to $s_{k+1}$ as follows:
  	\begin{itemize*}
  		\item If $e_{k+1}=e(1,\theta_{k+1})$ is a right vector, we start at $s_k$ and go along a counterclockwise arc on a unit circle, with initial tangent vector pointing in the positive $y$-direction, to $s_{k+1}$.
  		\item If $e_{k+1}=e(-1,\theta_{k+1})$ is a left vector, we start at $s_k$ and go along a clockwise arc on a unit circle, with initial tangent vector pointing in the positive $y$-direction, to $s_{k+1}$.
  	\end{itemize*}
  	In either case, the arc is less than a quarter circle since $0\leq\theta_i<90\degrees$.
  		
	\begin{figure}
	\centering
	\begin{tikzpicture}
		\path [fill=lightgray] (-4,4) arc [radius=1,start angle =90,end angle =0] arc [radius=1,start angle =90,end angle =0] arc [radius=1,start angle =90,end angle =0] arc [radius=1,start angle =90,end angle =0] arc [radius=1,start angle =180,end angle =90] arc [radius=1,start angle =180,end angle =90] arc [radius=1,start angle =180,end angle =90] arc [radius=1,start angle =180,end angle =90]
		-- cycle;
		\draw [help lines] (-4,0) grid (4,4);
		\draw [thick] (-4,0) -- (4,0);
		\draw (0,0) arc [radius=1,start angle =0,end angle =90] arc [radius=1,start angle =0,end angle =90] arc [radius=1,start angle =0,end angle =90] arc [radius=1,start angle =0,end angle =90];
		\draw (0,0) arc [radius=1,start angle = 180, end angle=90] arc [radius=1,start angle = 180, end angle=90] arc [radius=1,start angle = 180, end angle=90] arc [radius=1,start angle = 180, end angle=90];
		\draw [dashed,line width=1.5] (0,0) arc [radius=1,start angle =0,end angle =50] coordinate (s1) arc [radius=1,start angle =180,end angle =100] coordinate (s2) arc [radius=1,start angle =180,end angle =140] coordinate (s3) arc [radius=1,start angle =0,end angle =70] coordinate (s4) arc [radius=1,start angle =180,end angle =140];
		\draw [fill=black] (0,0) circle [radius=0.05] node [below] {$\langle R-L,0\rangle=s_0$};
		\draw [fill=black] (s1) circle [radius=0.05] node [right] {$s_1$};
		\draw [fill=black] (s2) circle [radius=0.05] node [right] {$s_2$};
		\draw [fill=black] (s3) circle [radius=0.05] node [right] {$s_3$};
		\draw [fill=black] (s4) circle [radius=0.05] node [left] {$s_4$};
	\end{tikzpicture}
	\caption{The staircase region at $\langle R-L,0\rangle$, with an  integer lattice added for scale. Using the Arc Fact iteratively to go from $s_k$ to $s_{k+1}$, we know that a concatenation of arcs on a unit circle of less than $90\degrees$ that starts at $\langle R-L,0\rangle$ as depicted by the dashed line can never leave the grey area.}
	\label{fig:arcfact}
	\end{figure}
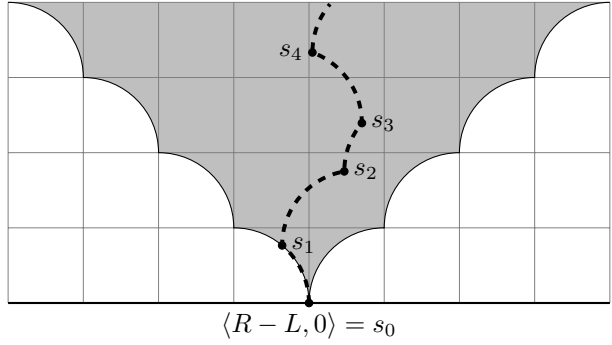
	
	We will for now suppose that the following fact is true and prove it later.

 	\textbf{Arc Fact:} A circular arc $\gamma(\theta)=\langle x,y\rangle+e(C,\theta)$ for $0\degrees\leq\theta< 90\degrees$ has the following properties for $C=\pm 1$ provided that $\gamma(0\degrees)$ is in the staircase region:
 	\begin{enumerate**}
 		\item $\gamma(\theta)$ is in the staircase region for any $0\degrees\leq\theta<90\degrees$.
 		\item $\gamma(0\degrees)$ is the only point on the arc that can lie in a corner of the staircase region.
 		\item If $\gamma(0\degrees)$ is not in one of the corners of the staircase region, then $\gamma(\theta)$ lies in the \emph{interior} of the staircase region for $0\degrees<\theta<90\degrees$.
 	\end{enumerate**}
 	The iterative application of this Arc Fact can be restated in a more intuitive way using figure \ref{fig:arcfact} as follows: Take a pen and start at $\langle R-L,0\rangle$. If now, one is only allowed to draw arcs along a (clockwise or counterclockwise) unit circle of less than $90\degrees$ (like the dashed line in the figure), one can never leave the grey area. Furthermore, once one is in the interior of the grey area (and therefore away from the corners), one is \enquote{stuck} in the interior and won't reach the boundary anymore.

  	Using these facts, we will now prove that $s_j=s$ must lie in the interior of the staircase region by studying the sequence $s_k$ as described above. First, we will prove the following claim by induction:
  	
  	\clearpage
  	\textbf{Claim:} $s_k$ lies in the staircase region at $\langle R-L,0\rangle$ for $k=0,1,\dots,j$ (so for now, it could be on the boundary).
  	
  	The claim is obvious for $k=0$ since $s_0=\langle R-L,0\rangle$. Assume it is true for given $k$ and $s_k$ is lying in the staircase region. We can define a path as follows:
  	\begin{align*}
  		\gamma(\theta)=\underbrace{\sum_{i\leq k}e(C_i,\theta_i)+\sum_{i\geq k+2}e(C_i,0)}_{=:\langle x,y\rangle}+e(C_{k+1},\theta)\qquad 0\degrees\leq\theta\leq\theta_{k+1}<90\degrees
  	\end{align*}
  	Observe that the we can apply (i) of the Arc Fact to this path:
  	\begin{itemize*}
  		\item $\gamma(0\degrees)=s_k$ is in the staircase region by hypothesis.
  		\item $0\degrees\leq\theta<90\degrees$ is given.
  		\item Therefore, $s_{k+1}=\gamma(\theta_{k+1})$ is in the staircase region by the Arc Fact.
  	\end{itemize*}
	
	The claim follows for all $k=0,1,\dots,j$. In particular $s=s_j$ is lying in the staircase region. We will now argue that it is in fact in the \emph{interior} of that region.
	
	
	Recall that $j\geq 4$, so besides $e_1=a$ and $e_2=c$ (that are left and right vectors respectively), there are at least two more vectors $e_3,\dots, e_j$. These $j-2$ vectors can be written as $e_i=e(C_i,\theta_i)$ for $\theta_i> 0\degrees$ ($a$ is horizontal, $c$ can be horizontal, but all other $j-2$ vectors must have positive $y$ component and therefore a positive angle).

	We conclude that there is at least two vectors with positive angle $\theta_i$. In that context, reconsider the Arc Fact and our sequence $s_k$ and note:
	\begin{itemize*}
		\item $s_2$ (which represents the sum after including $a$ and $c$) is in the staircase region but could be in a corner (in fact, it could be at $\langle R-L,0\rangle$ because both $a$ and $c$ could be horizontal vectors, so including their angles didn't actually change the sum).
		\item $s_3$ lies on an arc that starts at $s_2$ but reaches an angle $0<\theta_i<90\degrees$. Therefore, $s_3$ will not lie in a corner by (ii) of the Arc Fact.
		\item Applying (iii) of the Arc Fact to the arc from $s_3$ to $s_4$, note that $s_3$ does not lie in a corner. Therefore, $s_4$ lies in the interior of the staircase region.
		\item From now on, (iii) of the Arc Fact applies inductively: $s_k$ lies in the interior, i.e. not in a corner. Therefore, $s_{k+1}$ lies in the interior of the staircase region.
	\end{itemize*}
	
	It follows that $s=s_j$ is in fact lying in the \emph{interior} of the staircase region.
	
	Assume for the sake of contradiction that $\langle R-L,0\rangle\neq\langle0,0\rangle$. This means $s$ is lying in the interior of the staircase region starting at a point $\langle R-L,0\rangle\neq\langle 0,0\rangle$ on the integer lattice. At the same time, $s$ is a unit vector. However the unit circle centred at the origin does not intersect the interior of the staircase region at $\langle R-L,0\rangle \neq\langle 0,0\rangle$. This is a contradiction. It follows that $R-L=0$.
	
	We can now prove our initial claims (a') and (b')
	
	\begin{enumerate***}
		\item $R-L=0$, therefore $j=R+L=2R$ is even, proving the first claim.
		\item $R-L=0$ implies that $s$ lies on the intersection of the unit circle at the origin with the interior of the staircase region also starting at the origin. It follows that the angle between $s$ and the positive $y$-axis lies strictly within $\pm30\degrees$.
	\end{enumerate***}

	\clearpage

	We finish by proving the Arc Fact for $C=1$ (the case for $C=-1$ is just the mirror image). Let $\gamma(\theta)=\langle x,y\rangle+e(1,\theta)=\langle x,y\rangle+\langle\cos\theta,\sin\theta\rangle$ for $0\degrees\leq\theta< 90\degrees$. Note that we rotate \enquote{towards the left} as $\theta$ goes from $0\degrees$ to $90\degrees$. Also, the staircase region only grows wider as we go up. So even if $\gamma(0\degrees)$ is on the right staircase, no other point on $\gamma$ will be on or beyond the right staircase. We can therefore concentrate on the left staircase. Also, if we don't start on the left staircase but in the interior of the staircase region, this situation is just a right shift of an arc that starts on the staircase (compare the left and right of Figure \ref{fig:arcfactproof}). So it is enough to prove that an arc starting \emph{on} the staircase doesn't reach particular points.
	
	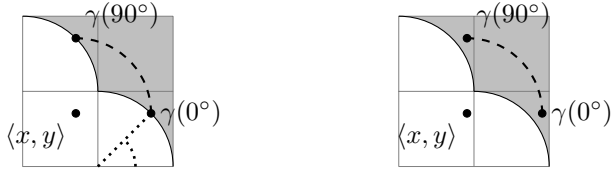
\begin{figure}
	\centering
	\begin{tikzpicture}
		\path [fill=lightgray] (1,-1) arc [radius=1,start angle = 0, end angle=90] arc [radius=1,start angle = 0, end angle=90] -- (1,1) -- cycle;
		\draw [help lines] (-1,-1) grid (1,1);
		\draw (1,-1) arc [radius=1,start angle = 0, end angle=90];
		\draw (0,0) arc [radius=1,start angle = 0, end angle=90];
		\path (0,-1) -- ++(45:1) -- ++(-1,0) node (xy) {};
		\path (xy.center) -- ++ (0:1) node (gamma0) {};
		\draw [thick,dashed] (gamma0) arc [radius=1,start angle = 0, end angle = 90] node [circle] (gamma90) {};
		\draw [fill=black] (gamma0) circle [radius=0.05];
		\draw [fill=black] (gamma90) circle [radius=0.05];
		\draw [fill=black] (xy) circle [radius=0.05];
		\node [right] at (gamma0) {$\gamma(0\degrees)$};
		\node [above right] at (gamma90) {$\gamma(90\degrees)$};
		\node [below left] at (xy) {$\langle x,y\rangle$};
		\draw [dotted,line width=1] (0.5,-1) arc [radius=0.5,start angle=0,end angle=45];
		\draw [dotted,line width=1] (0,-1) -- (gamma0.center);
	\begin{scope}[xshift=5cm]
		\path [fill=lightgray] (1,-1) arc [radius=1,start angle = 0, end angle=90] arc [radius=1,start angle = 0, end angle=90] -- (1,1) -- cycle;
		\draw [help lines] (-1,-1) grid (1,1);
		\draw (1,-1) arc [radius=1,start angle = 0, end angle=90];
		\draw (0,0) arc [radius=1,start angle = 0, end angle=90];
		\path (0,-1) -- ++(45:1) -- ++(0.2,0) -- ++(-1,0) node (xy) {};
		\path (xy.center) -- ++ (0:1) node (gamma0) {};
		\draw [thick,dashed] (gamma0) arc [radius=1,start angle = 0, end angle = 90] node [circle] (gamma90) {};
		\draw [fill=black] (gamma0) circle [radius=0.05];
		\draw [fill=black] (gamma90) circle [radius=0.05];
		\draw [fill=black] (xy) circle [radius=0.05];
		\node [right] at (gamma0) {$\gamma(0\degrees)$};
		\node [above right] at (gamma90) {$\gamma(90\degrees)$};
		\node [below left] at (xy) {$\langle x,y\rangle$};	
	\end{scope}
	\end{tikzpicture}
	\caption{Proof of the Arc Fact. On the left is the case where the arc $\gamma(\theta)$ (dashed) starts on the boundary of the staircase region, on the right is the case where the arc starts in the interior of the staircase region. Note that the latter is just a right shift of the situation where we start on the arc. The point in the centre of both pictures is chosen to be $\langle 0,0\rangle$. The dotted angle is $\varphi$.}
	\label{fig:arcfactproof}
	\end{figure}
	
	We will now prove the three parts of the Arc Fact for a path starting on the left staircase.
	
	\begin{enumerate**}
		\item Note that it is enough to consider two subsequent \enquote{steps} of the left staircase, namely the one at $\gamma(0\degrees)$ and the next higher one that $\gamma$ could possibly cross. It is obvious that $\gamma$ will not cross any steps on the left staircase that are further above or below. Shift the picture so that the corner between the two relevant steps is at $\langle 0,0\rangle$. This means we can write:
	\begin{align*}
		\gamma(0\degrees)=\langle0,-1\rangle+\langle\cos\varphi,\sin\varphi\rangle\qquad\text{for some }\varphi\in[0,90\degrees]
	\end{align*}
	And therefore
	\begin{align*}
		\gamma(\theta)&=\langle x,y\rangle+\langle\cos\theta,\sin\theta\rangle\\
		&=\gamma(0\degrees)-\langle1,0\rangle+\langle\cos\theta,\sin\theta\rangle\\
		&=\langle\cos\varphi+\cos\theta-1,\sin\varphi+\sin\theta-1\rangle
	\end{align*}
	To prove that $\gamma(\theta)$ lies on or to the right of the two steps, we need to show
	\begin{align*}
		\textrm{dist}(\langle-1,0\rangle,\gamma(\theta))\geq 1\quad \textrm{dist}(\langle0,-1\rangle,\gamma(\theta))\geq 1
	\end{align*}
	Note that
	\begin{align*}
		\textrm{dist}(\langle-1,0\rangle,\gamma(\theta))^2=(\cos\varphi+\cos\theta)^2+(\sin\varphi+\sin\theta-1)^2
	\end{align*}
	Basic two-variable calculus yields that the minimum of this function for $(\varphi,\theta)\in[0\degrees,90\degrees]\times[0\degrees,90\degrees]$ is in fact $1$. The same argument works for the second inequality.
	\item Note that, in the notation of (i), the only corners that $\gamma(\theta)$ can reach are at $\langle1,-1\rangle$, $\langle0,0\rangle$ and $\langle-1,1\rangle$. Reconsidering $\gamma(\theta)=\langle\cos\varphi+\cos\theta-1,\sin\varphi+\sin\theta-1\rangle$ and $(\varphi,\theta)\in[0\degrees,90\degrees]\times[0\degrees,90\degrees]$ and using calculus, these three points can only be reached if $(\varphi,\theta)$ is one of $(0\degrees,0\degrees)$, $(0\degrees,90\degrees)$, $(90\degrees,0\degrees)$ or $(90\degrees,90\degrees)$. Since $\theta<90\degrees$ by assumption, corners can therefore only be reached at $\theta=0\degrees$.
	\item Again we can consider $\gamma(\theta)=[\cos\varphi+\cos\theta-1,\sin\varphi+\sin\theta-1]$. If $\gamma(0\degrees)$ is not in a corner of the left staircase, we have $\varphi\neq0\degrees,90\degrees$. Applying calculus one more time, $\textrm{dist}(\langle-1,0\rangle,\gamma(\theta))^2>1$ and $\textrm{dist}(\langle0,-1\rangle,\gamma(\theta))^2>1$ for any $0\degrees<\theta<90\degrees$ and (iii) follows.
  	\end{enumerate**}
  	
  	This finishes the proof of the Arc Fact and therefore also concludes the proof of the Special Combined Angle Lemma.
\end{proof}

\vspace{-3mm}
\section{Global Properties on the plane}
\label{sect:global}

Using the local properties derived in the previous section, we now turn towards global properties of geodesic nets. For now, let $G$ be a geodesic net on the flat (zero curvature) plane. We will see later in section \ref{sect:negative} that these results readily extend to nonpositive curvature.

\subsection{The Convex Hull Property}

\begin{lemma}\label{thm:convexhull}
	Let $K$ denote the convex hull of all the unbalanced vertices in $G$. All balanced vertices lie in $K\setminus\partial K$.
\end{lemma}

\begin{proof}
	Assume there is a balanced vertex $v$ lying on or outside $\partial K$. This implies that we can draw a straight line through $v$ such that one side of that line is free of unbalanced vertices. Assume the line is vertical and all unbalanced vertices lie to the right of it. According to Lemma \ref{thm:straightline}, there must be an edge to the left of the line, leading to a vertex to the left of the line. It can't be unbalanced. Therefore, we can again draw a vertical line through that new vertex and get another vertex to the left of it. This process would continue ad infinitum, contradicting the finiteness of the geodesic net.
\end{proof}

\begin{lemma}\label{thm:atleastthree}
	A geodesic net that has at least one balanced vertex must have at least three unbalanced vertices.
\end{lemma}

\begin{proof}
	Otherwise the convex hull of the unbalanced vertices has empty interior. Apply Lemma \ref{thm:convexhull}.
\end{proof}

Note that the last Lemma rephrases the trivial cases of Theorem \ref{thm:main} for $n=0,1,2$.

\subsection{Paths and the Turn Angle}

From now on, we will frequently consider oriented paths using the following conventions:

\begin{itemize*}
	\item All paths that we consider are oriented, piecewise geodesic paths.
	\item A point on such a path that lies between two of its geodesic segments is called a \emph{vertex}.
	\item We will often refer to the geodesic segments of a path as \emph{edges}.
	\item For a path $\gamma$ that goes through a vertex $x$, we write $\gamma(\rightarrow x)$ for the restriction of $\gamma$ up to the point until it reaches $x$ for the first time and $\gamma(x\rightarrow)$ for the restriction of $\gamma$ starting at the point where it reached $x$ for the last time.
	\item For two paths or edges, we use $*$ as the symbol for concatenation.
	\item The notation $-\gamma$ refers to $\gamma$ with the opposite orientation.
	\item Given a closed path $\gamma$, we call the union of bounded components of $\mathbb{R}^2\setminus\gamma$ the \emph{inside} of $\gamma$ and the unbounded component the \emph{outside}.
\end{itemize*}

\begin{definition}
	Consider two consecutive edges $e,f$ on a path. The \emph{turn angle from $e$ to $f$ along the path} is defined as follows: If $e\neq -f$, the turn angle is the angle between the extension of $e$ to the other side of the vertex and the edge $f$ (By convention, a left turn is measured in positive angles and a right turn is measured in negative angles, see figure \ref{fig:turnangle}). If $e=-f$, the turn angle is $+180\degrees$ (so if we backtrack, this is considered a left turn).
\end{definition}

\begin{figure}[t]
\centering
\begin{tikzpicture}[scale=2]
	\draw [decoration={markings,mark=at position 1/2 with {\arrow{>[scale=2.5,
          length=2,
          width=3]}}},postaction={decorate}] (-1,0) -- node () [above=0.1] {$e$} (0,0);
	\draw [decoration={markings,mark=at position 3/4 with {\arrow{>[scale=2.5,
          length=2,
          width=3]}}},postaction={decorate}] (0,0) -- node () [above left=0.1] {$f$} ++(40:1);
	\draw [fill=black] (0,0) circle [radius=0.05];
	\draw [dashed] (0,0) -- (1,0);
	\draw (0.5,0) arc [radius = 0.5,start angle = 0, end angle = 40];
	
	\draw [decoration={markings,mark=at position 1/2 with {\arrow{>[scale=2.5,
          length=2,
          width=3]}}},postaction={decorate}] (2,0.65) -- node () [below=0.1] {$e$} (3,0.65);
	\draw [decoration={markings,mark=at position 3/4 with {\arrow{>[scale=2.5,
          length=2,
          width=3]}}},postaction={decorate}] (3,0.65) -- node () [below left =0.1] {$f$} ++(-40:1);
	\draw [fill=black] (3,0.65) circle [radius=0.05];
	\draw [dashed] (3,0.65) -- (4,0.65);
	\draw (3.5,0.65) arc [radius = 0.5,start angle = 0, end angle = -40];
\end{tikzpicture}
\caption{A positive turn angle (left) and a negative turn angle (right). If $e=f$, the turn angle is $+180\degrees$ by convention.}
\label{fig:turnangle}
\end{figure}
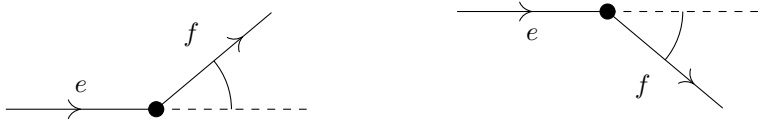

For further clarification, note that if $e$ and $f$ lie on a path that circumscribes a polygon in counterclockwise direction (this means that no backtracking is happening), the turn angle is exactly what is known as the exterior angle at the vertex of a polygon.

\begin{definition}
	Consider a path starting on an edge $e$ and ending on an edge $f$. We define the \emph{turn angle from $e$ to $f$ along the path} as the sum of all turn angles at the vertices between $e$ and $f$.
\end{definition}

Recall that all paths that we consider are piecewise geodesic. So the following well-known version of Gau\ss-Bonnet applies:
\begin{lemma}[Gau\ss-Bonnet, simple closed paths in the flat plane]\label{thm:gaussbonnetoriginal}
	If $e$ is an edge on a simple closed counterclockwise path, the turn angle from $e$ to $e$ along $\gamma$ is $360\degrees$.
\end{lemma}

As it turns out, we will need to use this fact in a context where $\gamma$ is not simple. We will carefully allow some exceptions to the requirement of simplicity, ensuring that Gau\ss-Bonnet still applies. To do so, we will define what it means for a path to be \emph{essentially simple}, using the notions of \emph{admissible backtracks} and \emph{non-transversal crossings}.

\begin{definition}[Admissible Backtrack]\label{def:admissiblebacktrack}
	Consider a path $\gamma$ doing a backtrack along an edge as follows:
	\begin{align*}
		e*a*(-a)*f
	\end{align*}
	Then this backtrack is \emph{admissible} if $f\neq -e$ and $a$ lies to the right of the path $e*f$. Otherwise it is \emph{inadmissible}. See figure \ref{fig:backtrack}.
\end{definition}

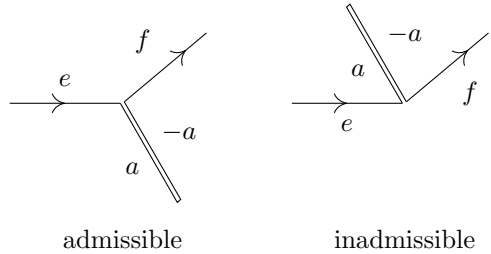
\begin{figure}
	\centering
	\begin{tikzpicture}[scale=1.5]
	\draw [decoration={markings,mark=at position 1/2 with {\arrow{>[scale=2.5,
          length=2,
          width=3]}}},postaction={decorate}] (-1,0) -- node () [above=0.1] {$e$} (180:0.02);
    \draw (180:0.02) -- node [below left] (a) {$a$} ($(300:1)+(220:0.02)$) -- ($(300:1)-(220:0.02)$) -- node [above right] (ma) {$-a$} (40:0.02);
	\draw [decoration={markings,mark=at position 3/4 with {\arrow{>[scale=2.5,
          length=2,
          width=3]}}},postaction={decorate}] (40:0.02) -- node () [above left=0.1] {$f$} ++(40:0.95);
	\node at (270:1.2) {admissible};
    \begin{scope}[xshift=2.5cm]
    	\draw [decoration={markings,mark=at position 1/2 with {\arrow{>[scale=2.5,
          length=2,
          width=3]}}},postaction={decorate}] (-1,0) -- node () [below=0.1] {$e$} (180:0.02);
    \draw (180:0.02) -- node [below left] (a) {$a$} ($(120:1)+(220:0.02)$) -- ($(120:1)-(220:0.02)$) -- node [above right] (ma) {$-a$} (40:0.02);
	\draw [decoration={markings,mark=at position 3/4 with {\arrow{>[scale=2.5,
          length=2,
          width=3]}}},postaction={decorate}] (40:0.02) -- node () [below right=0.1] {$f$} ++(40:0.95);	
    \node at (270:1.2) {inadmissible};
    \end{scope}
	\end{tikzpicture}	
	\caption{Two examples of backtracks}
	\label{fig:backtrack}
\end{figure}

When using the results of this section in the proofs of lemmas \ref{thm:60degrees}, \ref{thm:threevertices}, \ref{thm:60degreesnegative} and \ref{thm:threeverticesnonpos}, we will see that the only backtracks that are happening are admissible backtracks. That means neither will we have backtracks that lie to the left of the path, nor will we have \enquote{double backtracks} of the form $e*a*b*(-b)*(-a)*f$.

Based on this definition, the following lemma is apparent from figure \ref{fig:backtrack} and the fact that a backtrack is considered a turn of $+180\degrees$.
\begin{lemma}\label{thm:backtrackdontmatter}
	Consider an admissible backtrack $e*a*(-a)*f$. Then the turn angle along $e*f$ is the same as the turn angle along $e*a*(-a)*f$.
\end{lemma}

We will now specify what kind of crossing of paths we allow.
\begin{definition}[Non-Transversal Crossing]\label{def:nontransversal}
	 Consider a non-closed, simple path $e_1*\cdots * e_n$ ($n\geq 0$) and two paths $\alpha=a*e_1*\cdots*e_n*b$ and $\gamma=c*(-e_n)*\cdots*(-e_1)*d$ with $a\neq -d$ and $b\neq -c$ (if $n=0$, this means that $\alpha=a*b$ and $\gamma=c*d$ go through a common vertex). It follows that $a$, $b$, $c$, $d$ are arranged around the path $e_1*\cdots*e_n$ (or their common vertex in the case $n=0$). We say that $\alpha$ and $\gamma$ cross non-transversally if the edges are arranged counterclockwise in the order $abcd$ or $dcba$. If $n=0$ and either or both of the paths are backtracks (i.e. $a=-b$ or $c=-d$), it is still considered a non-transversal crossing.
\end{definition}


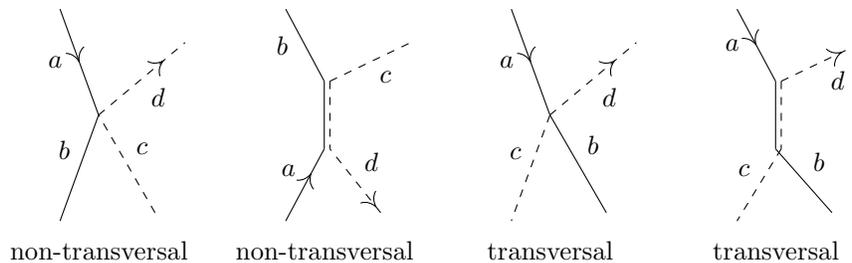
\begin{figure}[b]
	\centering
	\begin{tikzpicture}[scale=1.5]
	\draw [decoration={markings,mark=at position 1/2 with {\arrow{>[scale=2.5,
          length=2,
          width=3]}}},postaction={decorate}] (110:1) -- node () [left=0.1] {$a$} (180:0);
    \draw (180:0) -- node [above left] (a) {$b$} (250:1);
    \draw [dashed] ($(300:1)$) -- node [above right] (ma) {$c$} (40:0);
	\draw [dashed,decoration={markings,mark=at position 3/4 with {\arrow{>[scale=2.5,
          length=2,
          width=3]}}},postaction={decorate}] (40:0) -- node () [below right] {$d$} ++(40:1);	
    \node at (270:1.2) {non-transversal};
    \begin{scope}[xshift=2cm]
    \draw (110:1) -- node () [left=0.1] {$b$} (90:0.3);
    \draw (90:0.3) -- (270:0.3);
    \draw[decoration={markings,mark=at position 1/2 with {\arrow{<[scale=2.5,
          length=2,
          width=3]}}},postaction={decorate}] (270:0.3) -- node [above left] (a) {$a$} (250:1);
    \draw [dashed,decoration={markings,mark=at position 1/4 with {\arrow{<[scale=2.5,
          length=2,
          width=3]}}},postaction={decorate}] (300:1) -- node [above right] (ma) {$d$} ($(270:0.3)+(0:0.05)$);
    \draw [dashed] ($(90:0.3)+(0:0.05)$) -- ($(270:0.3)+(0:0.05)$);
	\draw [dashed] ($(90:0.3)+(0:0.05)$) -- node () [below right] {$c$} (40:1);	
    \node at (270:1.2) {non-transversal};	
    \end{scope}
    \begin{scope}[xshift=4cm]
    \draw [decoration={markings,mark=at position 1/2 with {\arrow{>[scale=2.5,
          length=2,
          width=3]}}},postaction={decorate}] (110:1) -- node () [left=0.1] {$a$} (180:0);
    \draw [dashed] (180:0) -- node [above left] (a) {$c$} (250:1);
    \draw (300:1) -- node [above right] (ma) {$b$} (40:0);
	\draw [dashed,decoration={markings,mark=at position 3/4 with {\arrow{>[scale=2.5,
          length=2,
          width=3]}}},postaction={decorate}] (40:0) -- node () [below right] {$d$} (40:1);
    \node at (270:1.2) {transversal};
    \end{scope}
    \begin{scope}[xshift=6cm]
    \draw [decoration={markings,mark=at position 1/2 with {\arrow{>[scale=2.5,
          length=2,
          width=3]}}},postaction={decorate}] (110:1) -- node () [left=0.1] {$a$} (90:0.3);
    \draw (90:0.3) -- (270:0.3);
    \draw [dashed] ($(270:0.3)+(0:0.05)$) -- node [above left] (a) {$c$} (250:1);
    \draw (300:1) -- node [above right] (ma) {$b$} (270:0.3);
    \draw [dashed] ($(90:0.3)+(0:0.05)$) -- ($(270:0.3)+(0:0.05)$);
	\draw [dashed,decoration={markings,mark=at position 3/4 with {\arrow{>[scale=2.5,
          length=2,
          width=3]}}},postaction={decorate}] ($(90:0.3)+(0:0.05)$) -- node () [below right] {$d$} (40:1);	
    \node at (270:1.2) {transversal};	
    \end{scope}
	\end{tikzpicture}	
	\caption{Examples of crossings. We will only allow non-transversal crossings. Note that our definition of non-transversal is quite strict. For example, the counterclockwise order $abdc$ (not depicted) is not called non-transversal according to our definition (it will, however, never occur below).}
	\label{fig:transcross}
\end{figure}

We can now define what it means for a path to be essentially simple:

\begin{definition}[essentially simple path]\label{def:essentiallysimple}
	We say that a path is \emph{essentially simple} if it is simple apart from the following two exceptions:
	\begin{itemize}
		\item It may contain admissible backtracks as defined above.
		\item It can revisit edges or vertices as long as this is a non-transversal crossing as defined above.
	\end{itemize}
\end{definition}

\begin{definition}\label{def:counterclockwise}
	If an essentially simple closed path $\gamma$ has the property that the outside (the unbounded component of $\mathbb{R}^2\setminus\gamma$) always lies to the right of $\gamma$, we call it a \emph{counterclockwise} path.
\end{definition}

Note that due to the presence of admissible backtracks and non-transversal intersections, there might be edges of an essentially simple path along which the outside lies simultaneously to the right \emph{and} the left of the path. The above definition allows for this to happen. If, on the other had, the outside were lying \emph{only} to the left (or on neither side) of at least one edge, the path would not be considered counterclockwise.

Based on this, we can rewrite Gau\ss-Bonnet from above:
\begin{lemma}[Gau\ss-Bonnet, essentially simple closed paths in the flat plane]\label{thm:gaussbonnetessential}
	If $e$ is an edge on an essentially simple closed counterclockwise path, the turn angle from $e$ to $e$ along $\gamma$ is $360\degrees$.
\end{lemma}

\begin{proof}
	If the path contains an admissible backtrack, due to lemma \ref{thm:backtrackdontmatter}, we can simply remove each such backtrack (even if it contains $e$) without changing the total turn angle.
		
	If the path has a non-transversal crossing, consider Figure \ref{fig:transcross} showing examples of the only two allowed arrangements of edges: The arrangement $dcba$ can't happen here since the path is counterclockwise (so the outside can't lie exclusively to the left at any edge). On the other hand, the arrangement $abcd$ (around a common vertex or common edges) can be realized as the local limit of a sequence of simple paths for which the outside still lies to the right of the path.
	
	We arrive at a sequence of simple closed counterclockwise paths $\gamma_i\to\gamma$. Gau\ss-Bonnet as in Lemma \ref{thm:gaussbonnetoriginal} applies to each $\gamma_i$ and therefore by continuity also to the limit $\gamma$.
\end{proof}

This lemma in turn allows us to prove the following:
\begin{lemma}[Conditional Path Independence]\label{thm:condpathind}
	Consider two paths $\gamma$ and $\delta$ with the same initial vertex $u$ and the same terminal vertex $v$ (i.e. $\gamma*(-\delta)$ is a closed path) as well as an edge $e$ incident to $u$ and an edge $f$ incident to $v$. If the following conditions are met, then the turn angle from $e$ to $f$ will be the same along $e*\gamma*f$ and $e*\delta*f$:
	\begin{enumerate*}
		\item Both $e$ and $f$ lie outside $\gamma*(-\delta)$ and they have no endpoints in common.
		\item Both $\gamma$ and $\delta$ are simple, except for admissible backtracks.
		\item If $\gamma$ and $-\delta$ meet anywhere except at their endpoints, it is a non-transversal crossing.
	\end{enumerate*}
	We call this the \emph{conditional path-independence} of the turn angle (see figure \ref{fig:condpathind}).
\end{lemma}

\begin{figure}[t]
\centering
\begin{tikzpicture}[scale=1.1]
    \draw (0,0) -- node [above] (e) {$e$} (2,0) node [above] (u) {$u$} -- ++(1,0.3) node (o) {} -- ++(0.7,0.3) node (p) {} -- ++(-0.5,1) -- ++(0.95,0) -- ++(0.1,-1) -- node (q) {} ++(0.05,0) -- ++(-0.1,1) -- ++(0.95,0) -- (6,2) node [above] (v) {$v$} -- node [above right] (f) {$f$} (7.3,1.8);
    \draw [dashed] (0,-0.05) -- (2,-0.05) -- ++(0.7,-0.4) -- ($(o.center)+(315:0.05)$) -- ($(p.center)+(315:0.05)$) -- ($(q.center)-(0,0.05)$) -- ++(1.3,0.5) -- (6,2-0.05) -- (7.3,1.8-0.05);
\end{tikzpicture}
\caption{Conditional path-independence: The turn angle from $e$ to $f$ is the same along either path. Note that both paths are going in a left-right direction in this picture.}
\label{fig:condpathind}
\end{figure}
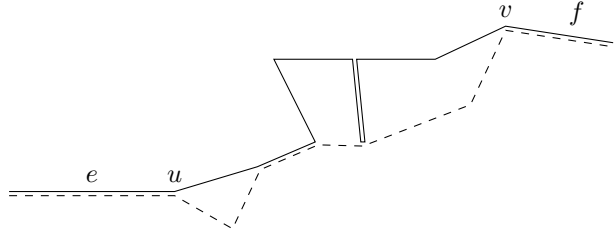

Before we get to the proof, it is worth pointing out that without the conditions, the turn angles along the two paths would only agree modulo $360\degrees$.

\begin{proof}
	First note that due to Lemma \ref{thm:backtrackdontmatter}, we can remove all admissible backtracks from $\gamma$ and $\delta$. Furthermore, in case $\gamma$ and $\delta$ agree on the first $i$ edges, we can write $\gamma=\epsilon*\gamma'$ and $\delta=\epsilon*\delta'$ for some path $\epsilon$ such that $\gamma'$ and $\delta'$ do not agree on their first edge. This reduces the proof to the question if $\gamma'$ and $\delta'$ produce the same turn angle. We therefore assume that $\gamma$ and $\delta$ do not agree on their first edge. For a similar reason, we assume that they don't agree on their last edge.
	
	$\gamma$ and $\delta$ are now both simple (since we removed all backtracks). That means that $-\gamma$ and $-\delta$ are also simple. Since $\gamma$ and $\delta$ don't agree on their first or last edge, the closed path $\gamma*(-\delta)$ also has no backtracks. Combining $\gamma$ and $-\delta$ can also not have produced transversal crossings by condition (c). It follows that $\gamma*(-\delta)$ is essentially simple. By the same arguments, $\delta*(-\gamma)$ is essentially simple.
	
	We can assume that $\gamma*(-\delta)$ is counterclockwise according to Definition \ref{def:counterclockwise}. This can be seen as follows: There must be some edge $e$ so that the outside is to the left or right of it (since the boundary of the unbounded component consists of edges of the path). If the outside is \emph{not} to the right of $e$, replace $\gamma*(-\delta)$ with $\delta*(-\gamma)$ and $e$ with $-e$. Therefore, after relabeling if necessary, we can assume that the outside is to the right of $e$. Starting at $e$, we follow the path $\gamma*(-\delta)$. Each of $\gamma$ and $-\delta$ is simple, therefore $\gamma*(-\delta)$ is simple except where $\gamma$ and $-\delta$ cross non-transversally. Refer to Figure \ref{fig:transcross} which demonstrates that if $\gamma$ and $\delta$ meet non-transversally and we arrive at the crossing with the outside to the right, we will also leave the crossing with the outside to the right (the outside would \emph{also} be to the left of the path during the crossing, which we allow). So the outside is always to the right of $\gamma*(-\delta)$ and therefore this is an essentially simple counterclockwise path.
	
	Now consider the closed path $\alpha=e*\gamma*f*(-f)*(-\delta)*(-e)$. Recall that $e$ and $f$ lie \emph{outside} $\gamma*(-\delta)$ and have no endpoints in common. This implies two things: (1) $\alpha$ is still a counterclockwise path. (2) $e$ and $f$ both lie to the right of the remainder of the path. The latter means that the two backtracks of $\alpha$ along $e$ and $f$ are admissible. Note that since $\gamma$ and $\delta$ don't agree on their first or last edges and have no backtracks themselves, there are no further backtracks. Therefore $\alpha$ is still essentially simple and Gau\ss-Bonnet as specified in Lemma \ref{thm:gaussbonnetessential} applies. Recalling that we consider backtracking to be a turn by $+180\degrees$ and setting the turn angle along $e*\gamma*f$ to be $x$ and the turn angle along $e*\delta*f$ to be $y$, we get
	\begin{align*}
		x+180\degrees-y+180\degrees=360\degrees
	\end{align*}
	Note that $\delta$ is free of backtracks, so the turn angle along $-\delta$ is in fact $-y$ since a turn where we run into issues with $\pm180\degrees$ doesn't happen.
	It follows that $x=y$.
\end{proof}

Another way of thinking of the turn angle along an essentially simple path that further illustrates the conditional path independence is: translate the initial edge $e$ of the path to an edge $e'$ that ends at the point where the terminal edge $f$ starts. The turn angle from $e'$ to $f$ at that point is now the same as the turn angle from $e$ to $f$ along the path.

We are now considering paths on a geodesic net, which are of course also piecewise geodesic paths.
\begin{definition}[First and Second Right Turn]\label{def:firstsecondturn} Consider a path through a balanced vertex.
	\begin{enumerate*}
		\item If the outgoing edge of the path immediately follows the incoming edge in counterclockwise order, we say that the path \emph{takes the first right turn}.
		\item If the outgoing edge of the path is the second edge following the incoming edge in counterclockweise order, we say that the path \emph{takes the second right turn}.
	\end{enumerate*}
\end{definition}
An example can be seen in figure \ref{fig:combinedanglev2}. Note that whenever a path takes the first or second right turn at a balanced vertex (which always has at least degree $3$), it is not backtracking.

In the context of these definitions, we will revisit the Special Combined Angle Lemma \ref{thm:specialcombinedangle} and arrive at the following three lemmas:

\begin{lemma}[First Turn Lemma]\label{thm:firsturn}
	If a path takes the first right turn at a balanced vertex, the turn angle is negative.
\end{lemma}

This Lemma is a direct consequence of Lemma \ref{thm:lessthan180}.

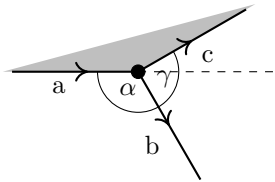
\begin{figure}[b]
\centering
\begin{tikzpicture}[scale=1.8]
	\node at (180:1) (aend) {};
	\node at (300:1) (bend) {};
	\node at (30:1) (cend) {};
	\path [fill=lightgray] (0,0) -- (cend.center) -- (aend.center) -- cycle;
	\draw [thick,decoration={markings,mark=at position 1/2 with {\arrow{<[scale=2.5,
          length=2,
          width=3]}}},postaction={decorate}] (0,0) -- node [below left] {a} (aend);
	\draw [thick,decoration={markings,mark=at position 1/2 with {\arrow{>[scale=2.5,
          length=2,
          width=3]}}},postaction={decorate}] (0,0) -- node [below left] {b} (bend);
	\draw [thick,decoration={markings,mark=at position 1/2 with {\arrow{>[scale=2.5,
          length=2,
          width=3]}}},postaction={decorate}] (0,0) -- node [below right] {c} (cend);
	\draw [fill=black] (0,0) circle [radius=0.05];
	\draw (180:0.3) arc [radius=0.3,start angle = 180,end angle = 300];
	\draw (300:0.3) arc [radius=0.3,start angle = 300,end angle = 390];
	\node at (240:0.15) {$\alpha$};
	\node at (345:0.2) {$\gamma$};
	\draw [dashed] (0,0) -- (1,0);
\end{tikzpicture}
\caption{Note that the turn angle from $a$ to $b$ (first right turn) and from $a$ to $c$ (second right turn) is measured in reference to the dashed line and compare with figure \ref{fig:combinedangle}.}
\label{fig:combinedanglev2}
\end{figure}

\begin{lemma}[Second Turn Lemma]\label{thm:secondturn}
	Consider a balanced vertex with incident edges $a,b,c$ following counterclockwise directly in that order. If the path from $a$ to $c$ (i.e. a path that takes the second right turn) has a positive turn angle as depicted in figure \ref{fig:combinedanglev2}, then:
	\begin{itemize*}
		\item the turn angle from $a$ to $c$ lies in $(0\degrees,60\degrees]$,
		\item the turn angle from $a$ to $b$ lies in $(-120\degrees,-60\degrees]$, and
		\item the turn angle from $-b$ to $c$ lies in $(-120\degrees,-60\degrees]$.
	\end{itemize*}
\end{lemma}

Compare figure \ref{fig:combinedanglev2} with figure \ref{fig:combinedangle} and it is immediate that this lemma is a reformulation of Lemma \ref{thm:combinedangle} (for vertices of degree $3$ and $4$) and Lemma \ref{thm:specialcombinedangle} (for vertices of degree $5$ or more). More specifically, for a vertex of degree $3$, the turn angles reach the extremal cases of $60\degrees$, $-60\degrees$, $-60\degrees$ respectively. For a vertex of degree $4$, the lemma is vacuously true since the turn angle along the second right turn can never be positive (it will, in fact, always be zero). For a vertex of degree $5$ or more, the turn angles lie in the interior of the three given intervals which follows from Lemma \ref{thm:specialcombinedangle}.

The situation in the following lemma is visualized in figure \ref{fig:60degrees}.

\begin{lemma}[$60\degrees$ Lemma, flat version]\label{thm:60degrees}
	Consider a path going through four edges $a,b,c,d$ of a geodesic net on the Euclidean plane and three vertices $u,v,w$ in the order $a,u,b,v,c,w,d$. Assume that
	\begin{enumerate*}
		\item $a\neq -b$, $c\neq -d$
		\item $u$ and $w$ are balanced ($v$ can be balanced or not).
		\item $b$ immediately follows $a$ at $u$ (i.e. $a*b$ takes the first right turn).
		\item $d$ immediately follows $c$ at $w$ (i.e. $c*d$ takes the first right turn).
		\item The convex hull of $u,v,w$ contains no unbalanced vertices (except, possibly, $v$ itself).
	\end{enumerate*}
	Then the turn angle from $a$ to $d$ along that path is at most $60\degrees$.
\end{lemma}

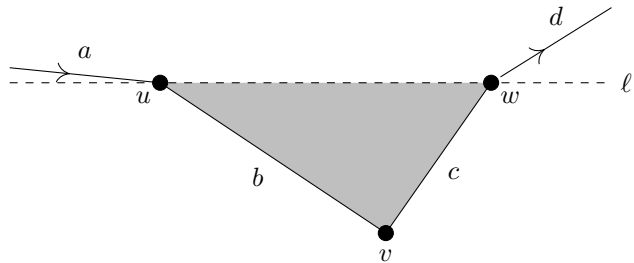
\begin{figure}[b]
\centering
\begin{tikzpicture}[scale=2]
	\path [fill=lightgray] (-1.5,1) -- (0,0) -- (0.7,1) -- cycle;
	\draw [decoration={markings,mark=at position 2/5 with {\arrow{>[scale=2.5,
          length=2,
          width=3]}}},postaction={decorate}] (-2.5,1.1) -- node [above=0.1cm] {$a$} (-1.5,1) node (u) {};
    \draw (u.center) -- node [below left] {$b$} (0,0) node (v) {} -- node [below right] {$c$} (0.7,1) node (w) {};
    \draw [decoration={markings,mark=at position 2/5 with {\arrow{>[scale=2.5,
          length=2,
          width=3]}}},postaction={decorate}] (w) -- node [above=0.1cm] {$d$} (1.5,1.5);
	\draw [fill=black] (u) circle [radius=0.05] node [below left] {$u$};
	\draw [fill=black] (v) circle [radius=0.05] node [below=0.1cm] {$v$};
	\draw [fill=black] (w) circle [radius=0.05] node [below right] {$w$};
	\draw [dashed] (-2.5,1) -- (1.5,1);
	\node at (1.5,1) [right] {$\ell$};
\end{tikzpicture}
\caption{Setup of the $60\degrees$ Lemma \ref{thm:60degrees}. There are no unbalanced vertices in the grey area. The line $\ell$ is used for Case 2 of the proof.}
\label{fig:60degrees}
\end{figure}

Note that by definition, $v$ is different from both $u$ and $w$. However, we allow $b=-c$ which then implies $u=w$. Note that we can't have both $b=-c$ and $a=-d$ because then $u=w$ would be a degree $2$ vertex which, for balanced vertices, is not possible.

The result of the lemma can be reformulated in the following way: If $a$ is translated to an edge $a'$ ending at $w$, the turn angle from $a'$ to $d$ is at most $60\degrees$, i.e. the clockwise angle from $a'$ to $d$ is at least $120\degrees$.

\begin{proof}
	First note the following two trivial cases:
	\begin{itemize*}
		\item If $u=w$, then the combined angle of $b=-c$ at $u=w$ is at most $240\degrees$ according to Lemma \ref{thm:combinedangle}. The turn angle from $a$ to $d$ is the combined angle minus $180\degrees$. The claim follows.
		\item If the turn angle from $b$ to $c$ at the vertex $v$ is nonpositive, note that the turn angles from $a$ to $b$ and from $c$ to $d$ are negative. Therefore the sum of all three is negative and the claim follows.
	\end{itemize*}
	
	We can therefore assume that $u\neq w$ and that the turn angle from $b$ to $c$ is positive. This means that the line $\ell$ through $u$ and $w$ is well-defined and unique and that we can rotate the picture such that $\ell$ is horizontal, $u$ is to the left of $w$ and $v$ is below $\ell$. So figure \ref{fig:60degrees} does in fact describe the only interesting situation (however $a$ or $d$ could also be below $\ell$ unlike in the figure).

	We define a path $\gamma$ starting at $u$ as follows:
	\begin{itemize*}
		\item Counting from $a$, we take the second right turn at $u$ (see Definition \ref{def:firstsecondturn}) and leave $u$ along that edge. This is the first edge of $\gamma$.
		\item From now on, always take the first right turn, unless that edge would lead to $v$. In that case, take the second right turn.
		\item Terminate as soon as the path either reaches $w$ or as soon as it left the convex hull of $u,v,w$ (this might mean that $\gamma$ only consists of a single edge starting at $u$).
	\end{itemize*}
	
	We first need to argue that this path is well-defined. In fact, there are no unbalanced vertices in the convex hull of $u,v,w$, so as long as we don't leave it (at which point the path terminates), we only reach balanced vertices and therefore a third edge in case we need it (to avoid $v$) always exists. Note that the same argument also implies that $\gamma$ doesn't backtrack.
	
	To show that the path terminates, assume that $\gamma$ never leaves the convex hull of $u,v,w$ but also never reaches $w$. Due to the finiteness of the geodesic net, it must eventually return to a vertex it has visited before. Note that $\gamma$ cannot have visited an unbalanced vertex since it never left the convex hull of $u,v,w$ and by definition never reaches $v$ (the only possibly unbalanced vertex in the convex hull). $\gamma$ therefore includes a closed polygon of balanced vertices which is contained inside the convex hull of $u,v,w$ in its entirety. Assume that this polygon is travelled counterclockwise by $\gamma$ (the clockwise case follows a very similar argument). Consider the rays from $v$ through each of the vertices of that polygon. One of these rays is the furthest to the right. Denote the vertex on the polygon that the ray reaches by $x$ (if there is more than one such vertex, pick the first one the ray reaches). The incoming edge of the path $\gamma$ is reaching $x$ from the left of the ray. Because $x$ is balanced, there must be an edge leaving $x$ to the right of the ray (Lemma \ref{thm:straightline}) and that edge cannot reach $v$. Instead of travelling counterclockwise along the polygon, $\gamma$ must therefore have taken one of the edges to the right of the ray, contradicting that $x$ is on the rightmost ray as described above. Therefore, the path indeed terminates at $w$ or leaves the convex hull of $u,v,w$ and is therefore well-defined.
	
	We now work on a case-by-case basis.
	
	\textbf{Case 1 $\gamma$ reaches $w$.}
	
	An example of $\gamma$ is given in figure \ref{fig:60degreesgamma}. Note that the path $a*\gamma*d$ and the path $a*b*c*d$ fulfil the requirements of conditional path independence as specified in Lemma \ref{thm:condpathind}. In fact:
	\begin{enumerate*}
		\item Due to $a*b$ taking the first right turn, $a$ lies outside the convex hull and so does $d$. The path $b*c*(-\gamma)$ on the other hand lies in the convex hull. Also, since we consider $u\neq w$, the edges $a$ and $d$ could only share the other endpoint which would contradict the first right turn conditions.
		\item We excluded the trivial case where $b=-c$, so $b*c$ is simple. We argued above that $\gamma$ can never return to a vertex. It follows that $\gamma$ is simple.
		\item $\gamma$ can't use $b$ or $c$ (it avoids any edges incident to $v$ by definition). Therefore, $b*c$ and $-\gamma$ never meet except at the endpoints.
	\end{enumerate*}
	So conditional path independence applies and the claim follows if the turn angle along $a*\gamma*d$ is at most $60\degrees$.
	
	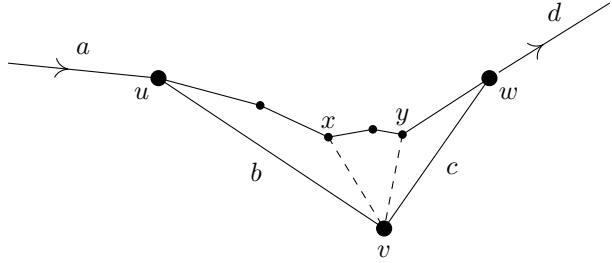
\begin{figure}[t]
	\centering
	\begin{tikzpicture}[scale=2]
		\draw [decoration={markings,mark=at position 2/5 with {\arrow{>[scale=2.5,
	          length=2,
	          width=3]}}},postaction={decorate}] (-2.5,1.1) -- node [above=0.1cm] {$a$} (-1.5,1) node (u) {};
	    \draw (u.center) -- node [below left] {$b$} (0,0) node (v) {} -- node [below right] {$c$} (0.7,1) node (w) {};
	    \draw [decoration={markings,mark=at position 2/5 with {\arrow{>[scale=2.5,
	          length=2,
	          width=3]}}},postaction={decorate}] (w) -- node [above=0.1cm] {$d$} (1.5,1.5);
		\draw [fill=black] (u) circle [radius=0.05] node [below left] {$u$};
		\draw [fill=black] (v) circle [radius=0.05] node [below=0.1cm] {$v$};
		\draw [fill=black] (w) circle [radius=0.05] node [below right] {$w$};
		\draw (u.center) -- ++(-15:0.7) node (r) {} -- ++ (-25:0.5) node (x) {} -- ++ (10:0.3) node (s) {} -- ++ (-10:0.2) node (y) {} -- (w.center);
		\draw [fill=black] (r) circle [radius=0.025] node {};
		\draw [fill=black] (x) circle [radius=0.025] node [above] {$x$};
		\draw [fill=black] (s) circle [radius=0.025] node {};
		\draw [fill=black] (y) circle [radius=0.025] node [above] {$y$};
		\draw [dashed] (x.center) -- (v.center);
		\draw [dashed] (y.center) -- (v.center);
	\end{tikzpicture}
	\caption{An example for the situation of Case 1. The upper path is $\gamma$. Note that the turn angle from $a$ to $d$ will be the same, no matter which of the paths we take. The dashed lines are $e_v$ and $f_v$ respectively.}
	\label{fig:60degreesgamma}
	\end{figure}

	If this angle is nonpositive, there is nothing to show. If it is positive, then at some point the turn angle at a vertex on $a*\gamma*d$ must be positive. Denote by $x$ the first edge with a positive turn angle and by $y$ the last edge with a positive turn angle. Note that the following argument still works in case any (or several) of $u,w,x,y$ coincide.
	
	At $x$: By the First Turn Lemma \ref{thm:firsturn}, $a*\gamma*d$ must have taken the second right turn at $x$. Therefore, $x$ is adjacent to $v$. Denote the incoming edge to $x$ along the path by $e_{in}$ and the edge from $x$ to $v$ by $e_{v}$. Now, by the Second Turn Lemma \ref{thm:secondturn}, the turn angle from $e_{in}$ to $e_{v}$ is in $(-120\degrees,-60\degrees]$.

	At $y$: By the First Turn Lemma \ref{thm:firsturn}, $a*\gamma*d$ must have taken the second right turn at $y$. Therefore, $y$ is adjacent to $v$. Denote the outgoing edge from $y$ along the path by $f_{out}$ and the edge from $v$ to $y$ by $f_{v}$. Now, by the Second Turn Lemma \ref{thm:secondturn}, the turn angle from $f_{v}$ to $f_{out}$ is in $(-120\degrees,-60\degrees]$.
	
	Consider the path $\gamma':=\gamma(\rightarrow x)* e_v* f_v* \gamma(y\rightarrow)$. Again, conditional path independence applies to $a*\gamma'*d$ and $a*b*c*d$. Let's check the three conditions:
	\begin{enumerate*}
		\item $a$ and $d$ still lie outside the convex hull and don't share any endpoints.
		\item $b*c$ is still simple. As long as it follows $\gamma$, the path $\gamma'$ must be simple. It could only fail to be simple at $e_v*f_v$ in the case that these two edges coincide. In that case it would be an admissible backtrack, since by definition (first right turns), $e_v$ and $f_v$ lie to the right of the remainder of $\gamma'$.
		\item Note that $b*c$ is on the boundary of the convex hull whereas $-\gamma'$ is in the convex hull. So $b*c$ and $-\gamma'$ could meet at edges (for example, if $x=u$, then $e_v=b$) but for $-\gamma'$ to cross $b*c$ transversally to the other side (to produce anything but the the non-transversal counterclockwise orders as given in definition \ref{def:nontransversal}), it would have to leave the convex hull which it doesn't.
	\end{enumerate*}
	\begin{samepage}
	It follows that it is enough to show that the turn angle along $a*\gamma'*d$ is $60\degrees$ or less. And in fact note:
	\begin{itemize*}
		\item At $x$, the turn angle is in $(-120\degrees,-60\degrees]$.
		\item At $v$, the turn angle is at most $180\degrees$.
		\item At $y$, the turn angle is in $(-120\degrees,-60\degrees]$.
		\item All other turn angles are (by the choice of $x$ and $y$) negative.
	\end{itemize*}
	Since the turn angle along $a*\gamma'*d$ is the sum of these angles, the claim follows.
	\end{samepage}

	\textbf{Case 2 $\gamma$ leaves the convex hull of $u,v,w$ and terminates.}
	
	\nopagebreak
	
	We now need to consider three subcases depending on the position of $a,d$ relative to the line $\ell$. Recall that $\ell$ is the horizontal line through $u$ and $w$. Denote by $\overline{\ell}$ the segment of $\ell$ between $u$ and $w$.
	
	\textbf{Case 2a both $a$ and $d$ lie above $\ell$.} Because $\gamma$ left the convex hull of $u,v,w$ and $a$ lies above $\ell$, at some vertex the turn angle along $a*\gamma$ must be positive. Pick the first vertex for which that is the case and call it $x$ (possibly this is $u$). Define $e_{in}$ and $e_v$ as above.
	
	We define a path $\epsilon$. Note that the following is the mirror image of the way we found $\gamma$: At $w$, counting from $d$, take the second left turn. Now always take the first left turn unless the edge leads to $v$ in which case we take the second left turn. The path terminates when reaching $u$ or when leaving the convex hull of $u,v,w$. The same arguments regarding the well-definedness still apply. But note that the resulting path can't actually reach $u$ or otherwise $\epsilon=-\gamma$ and $\gamma$ would have reached $w$, a case we already dealt with.
	
	If we consider $-(d*\epsilon)$ (i.e. $d*\epsilon$ with the opposite orientation), because $d*\epsilon$ left the convex hull of $u,v,w$ and $d$ lies above $\ell$, at some vertex the turn angle along $-(d*\epsilon)$ must be positive.  Pick the last vertex for which that is the case and call it $y$ (possibly this is $w$). Define $f_v$ and $f_{out}$ as above.
	
	We now have the exact same angle setup as in Case 1 and the claim follows.
	
	\textbf{Case 2b $d$ lies on or below $\ell$.} This is a rather pathological case that needs particular consideration.
	
	Assume that the turn angle along the path $a* b* c* d$ is more than $60\degrees$. This means that the turn angle along $a*\overline{\ell}*d$ is more than $60\degrees$ (In this case, it is obvious that conditional path independence applies to these two paths). However, since $d$ lies on or below $\ell$, the turn angle from $\overline{\ell}$ to $d$ must be nonpositive. It follows that the turn angle from $a$ to $\overline{\ell}$ must be more than $60\degrees$. In particular, the first edge of $\gamma$ must lie inside the convex hull, otherwise $a*\gamma$ would take a turn of more than $60\degrees$ at $u$, which contradicts the Second Turn Lemma \ref{thm:secondturn}.

	Consider the following path which we call $\gamma_+$: Truncate the path $\gamma$ so that it terminates at the point where it intersects $\overline{\ell}$ for the last time. This means it most likely won't terminate on a vertex but in the middle of an edge. Now continue along $\overline{\ell}$ to $w$.
	
	Note that the last edge of $\gamma$ is either on $\overline{\ell}$ or it starts below $\overline{\ell}$ and ends above it. So the last turn of $\gamma_+$ onto $\overline{\ell}$ must have been a right turn or no turn at all (nonpositive turn angle).
	
	The three conditions for conditional path independence are again met by $a*b*c*d$ and $a*\gamma_+*d$ (see the arguments in case 1). So the turn angles must be the same and therefore $a*\gamma_+*d$ also has a turn angle of more than $60\degrees$. But the last turn of $\gamma+$ had a nonpositive angle and the turn onto $d$ is also negative as argued above, so it follows that the turn angle along the path $\gamma$ must be more than $60\degrees$. So there must be at least one vertex along $\gamma$ with a positive turn angle. Call the first such vertex $x$ and the last such vertex $y$. As usual, we allow $x$ and $y$ to coincide.
	
	Using the same construction as before, we get a path $\gamma'=\gamma(\rightarrow x)* e_v* f_v* \gamma(y\rightarrow)$. Since $\gamma'$ and $\gamma$ agree on all edges apart from the middle section where $\gamma'$ lies to the right of $\gamma$ and has at most a single admissible backtrack (if $e_v=-f_v$), conditional path independence applies. But note that the turn angle along $\gamma$ is more than $60\degrees$, whereas, by the same arguments as in case 1, the turn angle along $\gamma'$ is at most $60\degrees$. This contradiction implies that our initial assumption that the turn angle along $a*b*c*d$ was more than $60\degrees$ must have been wrong.
	
	\textbf{Case 2c $a$ lies on or below $\ell$}. This case is simply the mirror image of Case 2b.
\end{proof}

\subsection{Circumference}

Recall that we are considering a geodesic net $G$ on the plane, which is a connected graph embedded in $\mathbb{R}^2$. In this context, note the following definition, which is quoted from \cite{MR2181153}, page 312.

\begin{definition}
	Consider a face $F$ of an embedded graph $G$. The \emph{boundary walk} of $F$ is a closed walk in $G$ that corresponds to a complete transversal of the perimeter of the polygonal region within the face. Note that vertices and edges can reoccur in a boundary walk. In particular, if both sides of an edge lie on a single region, the edge is retraced on the boundary walk.
\end{definition}

Using the concept of boundary walks, we can define the following:

\begin{definition}
	Since $G$ is a connected graph embedded in $\mathbb{R}^2$, it has exactly one outer face, which is the unbounded component of $\mathbb{R}^2\setminus G$. We call the boundary walk of this outer face the \emph{circumference} of $G$. By convention, we orient it counterclockwise, i.e. such that the outer face lies to the right of the curve. For an example, see figure \ref{fig:circumference}.
\end{definition}

Note that it is possible that the circumference, being a boundary walk, travels the same edge or vertex several times (as in the figure). It does not, however, cross itself transversally. Otherwise it would enclose other faces.

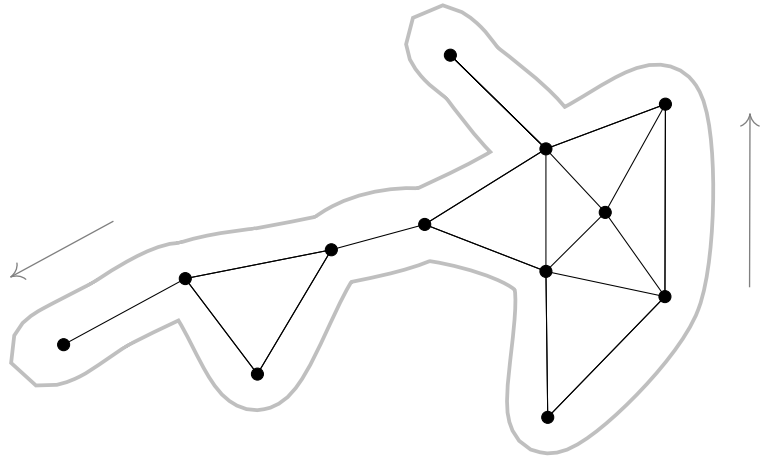
\begin{figure}
	\centering
	\begin{tikzpicture}[scale=0.8]
	\node at (2.0218511627906954,3.0989633488372066) (a) {};
	\node at (3.220370232558136,1.5109255813953466) (b) {};
	\node at (6.,4.) (c) {};
	\node at (10.,6.) (d) {};
	\node at (6.4264087441860385,6.814372465116275) (e) {};
	\node at (8.014446511627897,5.256297674418601) (f) {};
	\node at (8.014446511627897,3.2188152558139507) (g) {};
	\node at (9.992002976744173,2.799333581395346) (h) {};
	\node at (9,4.2) (z) {};
	\node at (8.044409488372084,0.7918141395348818) (i) {};
	\node at (4.448852279069762,3.578370976744183) (j) {};
	\node at (0.,2.) (k) {k};
	\path (a.center) -- node (startpoint) {} (j.center);
	\path (e.center) +(30:0.1) node (eneighbour) {};
	\path (k.center) +(130:0.1) node (kneighbour) {};
	\node at (a) [above left=1] (aarrownode) {};
	\node at (k) [above left=1] (karrownode) {};
	\node at (d) [right=1] (darrownode) {};
	\node at (h) [right=1] (harrownode) {};
	\begin{scope}
	\draw [line width=1cm,lightgray] plot [smooth] coordinates {(startpoint.center) (j.center) (c.center) (f.center) (e.center) (eneighbour.center) (f.center) (d.center) (h.center) (i.center) (g.center) (c.center) (j.center) (b.center) (a.center) (k.center) (kneighbour.center) (a.center) (startpoint.center)};
	\draw [line width=0.9cm,color=white] plot [smooth] coordinates {(startpoint.center) (j.center) (c.center) (f.center) (e.center) (eneighbour.center) (f.center) (d.center) (h.center) (i.center) (g.center) (c.center) (j.center) (b.center) (a.center) (k.center) (kneighbour.center) (a.center) (startpoint.center)};
	\draw [fill=white] (a.center) -- (j.center) -- (b.center) -- cycle;
	\draw [fill=white] (c.center) -- (f.center) -- (e.center) -- (f.center) -- (d.center) -- (h.center) -- (i.center) -- (g.center) -- cycle;
	\draw [decoration={markings,mark=at position 1 with {\arrow{>[scale=2.5,
          length=2,width=3]}}},postaction={decorate},color=gray,line width=0.5] (aarrownode) -- (karrownode);
	\draw [decoration={markings,mark=at position 1 with {\arrow{>[scale=2.5,
          length=2,width=3]}}},postaction={decorate},color=gray,line width=0.5] (harrownode) -- (darrownode);
	\end{scope}
	\draw [fill=black] (a) circle [radius=0.1];
	\draw [fill=black] (b) circle [radius=0.1];
	\draw [fill=black] (c) circle [radius=0.1];
	\draw [fill=black] (d) circle [radius=0.1];
	\draw [fill=black] (e) circle [radius=0.1];
	\draw [fill=black] (f) circle [radius=0.1];
	\draw [fill=black] (g) circle [radius=0.1];
	\draw [fill=black] (h) circle [radius=0.1];
	\draw [fill=black] (z) circle [radius=0.1];
	\draw [fill=black] (i) circle [radius=0.1];
	\draw [fill=black] (j) circle [radius=0.1];
	\draw [fill=black] (k) circle [radius=0.1];
	\draw (a.center) -- (b.center);
	\draw (a.center) -- (k.center);
	\draw (a.center) -- (j.center);
	\draw (b.center) -- (j.center);
	\draw (j.center) -- (c.center);
	\draw (c.center) -- (g.center);
	\draw (c.center) -- (f.center);
	\draw (f.center) -- (e.center);
	\draw (f.center) -- (d.center);
	\draw (f.center) -- (g.center);
	\draw (d.center) -- (h.center);
	\draw (g.center) -- (i.center);
	\draw (i.center) -- (h.center);
	\draw (h.center) -- (z.center);
	\draw (g.center) -- (h.center);
	\draw (g.center) -- (z.center);	
	\draw (f.center) -- (z.center);
	\draw (d.center) -- (z.center);
	\end{tikzpicture}
	\caption{Example for the circumference of a graph which is the path resulting from \enquote{shrinking} the grey curve onto the graph.}
	\label{fig:circumference}
\end{figure}

\begin{lemma}\label{thm:circumferencefirstrightturn}
	The circumference must always take the first right turn when going through a balanced vertex (see Definition \ref{def:firstsecondturn}).	
\end{lemma}

\begin{proof}
	If the circumference took any other edge but the immediately following edge $e$ in counterclockwise order at a vertex, note that $e$ would lie to the right of the circumference. However, by definition, the circumference is a boundary walk and can therefore never cut off any edges from the graph.
\end{proof}

An immediate consequence of the previous lemma and Lemma \ref{thm:firsturn} is:

\begin{lemma}\label{thm:outeranglecircumference}
	The turn angle at each balanced vertex on the circumference must be negative.
\end{lemma}

\begin{samepage}
We can now conclude:
\begin{lemma}
	The circumference includes unbalanced vertices.
\end{lemma}
\end{samepage}

\begin{proof}
	Otherwise the circumference would only visit balanced vertices. By the previous lemma, the turn angle at each vertex would therefore be negative. This implies that the turn angle along the whole circumference is negative. We now have a closed clockwise path such that the unbounded component of $\mathbb{R}\setminus G$ lies to the right of it and therefore inside a curve. This is a contradiction.
\end{proof}

\section{Proof for $f(3)=1$ on the flat plane}
\label{sect:proof3}

We can now prove the main result of Theorem \ref{thm:main} for the case of a geodesic net on the Euclidean plane. We will cover the case of nonpositive curvature in section \ref{sect:negative}.

\begin{lemma}\label{thm:threevertices}
	A geodesic net on the Euclidean plane with three unbalanced vertices has at most one balanced vertex.
\end{lemma}

\begin{proof}
	First, if the geodesic net $G$ includes any edges that start \textit{and} end at an unbalanced vertex, we remove them. This leads to two cases:
	\begin{enumerate*}
		\item We get a geodesic net $G'$ with fewer unbalanced vertices (since removing an incident edge might balance a previously unbalanced vertex). So $G'$ has one or two unbalanced vertices and therefore no balanced vertices as demonstrated previously. It follows that  $G$ also has no balanced vertices and we are done.
		\item We get a geodesic net $G'$ with the same number of unbalanced and balanced vertices as $G$. In this case, it suffices to study $G'$.
	\end{enumerate*}
	We can therefore assume that no such \enquote{irrelevant edges} exist.
	
	We can assume that the three unbalanced vertices $x,y,z$ are not collinear (otherwise the interior of their convex hull is empty and by Lemma \ref{thm:convexhull}, there are no balanced vertices). This implies that $x,y,z$ are arranged on a triangle.
	
	All balanced vertices must lie inside the triangle formed by these three points.
	
	\textbf{Claim:} We can assume that removing any one of $x,y,z$ must not disconnect $G$.\\
	If, say, removing $x$ would disconnect the geodesic net, consider the following process: Remove $x$, splitting $G$ into at least two connected components and add a copy of $x$ to each of them. Each of the resulting components must contain at least two unbalanced vertices (otherwise the component would have no vertices besides $x$ which isn't possible). But there are only three unbalanced vertices in total and at least two components. Therefore, each component has at most two unbalanced vertices and therefore no balanced vertices. We deduce that there is a total of zero balanced vertices.
	
	Denote the circumference of the geodesic net by the path $\gamma$.
	
	\textbf{Claim:} $\gamma$ travels through each of $x,y,z$ exactly once.\\
	First note that the circumference must reach each of the three at least once, since the geodesic net is connected and $x,y,z$ are on the boundary of the geodesic net. If it travelled through, say, $x$ twice, removing $x$ would split up $G$, a case we just excluded.
	
	This implies that $\gamma$ either visits the three unbalanced vertices in the order $x,y,z$ or in the order $x,z,y$. After relabeling if necessary, we assume the order is $x,y,z$.
	
	Therefore, the circumference travels along vertices in the following order.
	\begin{align*}
		x,u_1,\dots,u_r,y,v_1,\dots,v_s,z,w_1,\dots,w_t,x
	\end{align*}
	where all the $u_i,v_j,w_k$ are balanced vertices.
	
 	First note that on the circumference, two unbalanced vertices never follow directly because we would have an \enquote{irrelevant edge}. This implies $r,s,t\geq 1$. We will argue that $r=s=t=1$.
 	
 	Consider the neighbourhood of the vertex $y$ in the following sense (compare figure \ref{fig:labelthehex} to figure \ref{fig:60degrees}):
	\begin{itemize*}
		\item The edge entering $u_r$ is called $a$.
		\item $u_r$ is called $u$
		\item The edge from $u_r$ to $y$ is called $b$.
		\item $y$ is called $v$.
		\item The edge from $y$ to $v_1$ is called $c$.
		\item $v_1$ is called $w$.
		\item The edge leaving $v_1$ is called $d$.
	\end{itemize*}
	
	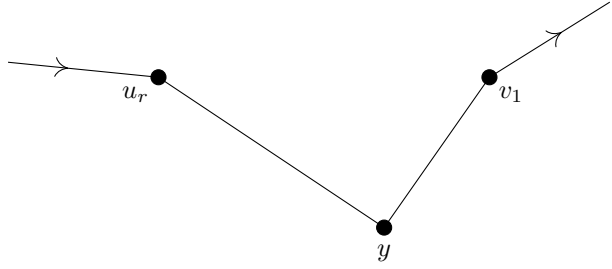
\begin{figure}
	\centering
	\begin{tikzpicture}[scale=2]
		\draw [decoration={markings,mark=at position 2/5 with {\arrow{>[scale=2.5,
	          length=2,
	          width=3]}}},postaction={decorate}] (-2.5,1.1) -- node [above=0.1cm] {} (-1.5,1) node (u) {};
	    \draw (u.center) -- node [below left] {} (0,0) node (v) {} -- node [below right] {} (0.7,1) node (w) {};
	    \draw [decoration={markings,mark=at position 3/5 with {\arrow{>[scale=2.5,
	          length=2,
	          width=3]}}},postaction={decorate}] (w.center) -- node [above=0.1cm] {} (1.5,1.5);
		\draw [fill=black] (u) circle [radius=0.05] node [below left] {$u_r$};
		\draw [fill=black] (v) circle [radius=0.05] node [below=0.1cm] {$y$};
		\draw [fill=black] (w) circle [radius=0.05] node [below right] {$v_1$};
	\end{tikzpicture}
	\caption{The circumference around the vertex $y$. Compare this figure to figure \ref{fig:60degrees}}
	\label{fig:labelthehex}
	\end{figure}

	Since there are only three unbalanced vertices, no unbalanced vertices are inside the convex hull of $u,v,w$.
	
	We now have the setup of the $60\degrees$ Lemma \ref{thm:60degrees} and can conclude that the  turn angle from $a$ to $d$ is at most $60\degrees$. In other words: The sum of the turn angles at $u_r$, $y$ and $v_1$ is at most $60\degrees$. We will abuse notation and use the names for the vertices also for the turn angles of the circumference at these vertices. So we can write:
	\begin{align}\label{eq:y}
		u_r+y+v_1\leq 60\degrees
	\end{align}	
	By analogous arguments around $x$ and $z$ we get
	\begin{align}
		w_t+x+u_1\leq 60\degrees\label{eq:x}\\
		v_s+z+w_1\leq 60\degrees\label{eq:z}
	\end{align}
	Note that the circumference is essentially simple as specified by definition \ref{def:essentiallysimple}:
	\begin{itemize*}
		\item Backtracking could only happen while visiting one of the unbalanced vertices, since at balanced vertices, the circumference always takes the first right turn (see Lemma \ref{thm:circumferencefirstrightturn}). If backtracking happened at, say, $x$, note that then $w_t$ and $u_1$ would be the same balanced vertex. $\gamma$ takes the first right turn both before and after visiting $x$, so double backtracking would imply that $u_1=w_t$ has degree $2$ which is impossible for a balanced vertex. Also, since we do right turns right before and after $x$, the edge to/from $x$ must lie to the right of the remainder of the circumference. We conclude that if backtracking happens, it is admissible.
		\item Since the circumference is a boundary walk, it can never cut off any edges, but this would be necessary for a transversal crossing.
	\end{itemize*}
	Since the circumference is also closed and the outside always lies to the right of it by definition, it is counterclockwise according to Definition \ref{def:counterclockwise} and Gau\ss-Bonnet as described in Lemma \ref{thm:gaussbonnetessential} applies:
	\begin{align}\label{eq:closed}
		x+y+z+\sum u_i+\sum v_j+\sum w_k=360\degrees
	\end{align}
	Now assume that $r>1$ and therefore $u_1$ and $u_r$ are indeed separate angles. Note that the turn angle at balanced vertices on the circumference must be negative and rewrite equation \eqref{eq:closed} to:
	\begin{align*}
		\underbrace{w_t+x+u_1}_{\leq 60\degrees}+\underbrace{u_r+y+v_1}_{\leq 60\degrees}+\underbrace{z}_{\leq 180\degrees}+\underbrace{\sum_{i\neq1,r} u_i+\sum_{j\neq1} v_j+\sum_{k\neq t} w_k}_{\leq 0}=360\degrees
	\end{align*}
	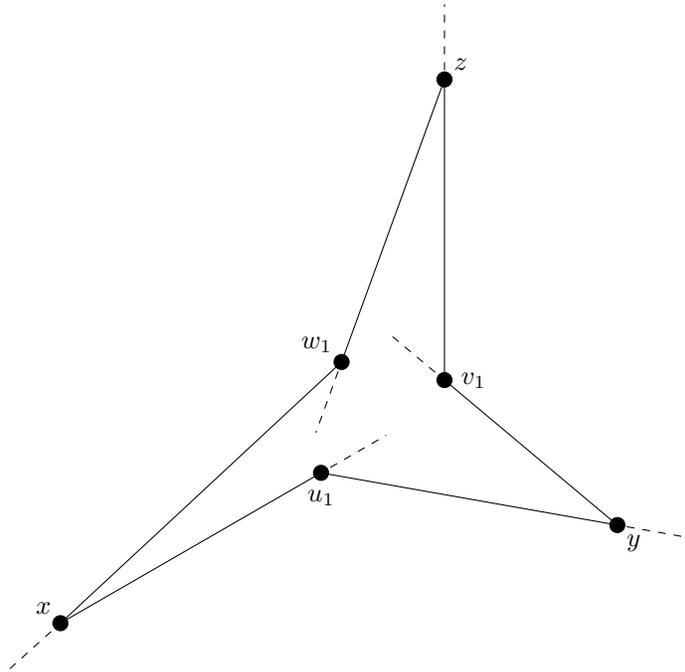
\begin{figure}
	\centering
	\begin{tikzpicture}[scale=2]
		\draw (0,0) node (x) {} -- ++(30:2) node (u) {} -- ++(-10:2) node (y) {} -- ++ (140:1.5) node (v) {}
		-- ++ (90:2) node (z) {} -- ++ (250:2) node (w) {} -- cycle;
		\draw [fill=black] (x) circle [radius=0.05];
		\draw [fill=black] (u) circle [radius=0.05];
		\draw [fill=black] (y) circle [radius=0.05];
		\draw [fill=black] (v) circle [radius=0.05];
		\draw [fill=black] (z) circle [radius=0.05];
		\draw [fill=black] (w) circle [radius=0.05];
		\node at (x) [above left] {$x$};
		\node at (u) [below=0.1cm] {$u_1$};
		\node at (y) [below right] {$y$};
		\node at (v) [right=0.1cm] {$v_1$};
		\node at (z) [above right] {$z$};
		\node at (w) [above left] {$w_1$};
		\draw [dashed] (x) -- ++(222:0.5);
		\draw [dashed] (u) -- ++(30:0.5);
		\draw [dashed] (y) -- ++(-10:0.5);
		\draw [dashed] (v) -- ++(140:0.5);
		\draw [dashed] (z) -- ++(90:0.5);
		\draw [dashed] (w) -- ++(250:0.5);
	\end{tikzpicture}
	\caption{The circumference of the geodesic net after we have established that $r=s=t=1$. The dashed lines are the lines of reference for the turn angles.}
	\label{fig:thehex}
	\end{figure}
	This is a contradiction. It follows that $r=1$. By analogous arguments $s=t=1$ and we have in fact the situation shown in figure \ref{fig:thehex} . This means that the three inequalities \eqref{eq:y}, \eqref{eq:x}, \eqref{eq:z} can be rewritten as:
	\begin{align*}
		u_1+y+v_1\leq 60\degrees\qquad
		w_1+x+u_1\leq 60\degrees\qquad
		v_1+z+w_1\leq 60\degrees
	\end{align*}
	Adding up and rearranging, we get
	\begin{align*}
		x+y+z+u_1+v_1+w_1\leq 180\degrees-(u_1+v_1+w_1)
	\end{align*}
	But because of \eqref{eq:closed}, we get
	\begin{align*}
		360\degrees=x+y+z+u_1+v_1+w_1\leq 180\degrees-(u_1+v_1+w_1)
	\end{align*}
	We see that $u_1+v_1+w_1\leq-180\degrees$. But then also $x+y+z\geq540\degrees$. Since each of $x,y,z$ can be at most $180\degrees$, it follows that
	\begin{align*}
		x=y=z=180\degrees
	\end{align*}
	That implies $u_1=v_1=w_1$ and therefore $G$ is just a tree with three unbalanced vertices and one degree three balanced vertex in the centre. So in fact, this tree is the only possible geodesic net with three unbalanced vertices that includes any balanced vertices.
\end{proof}

\vspace{-5mm}

\section{The case of nonpositive curvature}
\label{sect:negative}

In this section, we will establish that the main theorem \ref{thm:main} holds true for nonpositive curvature on $\mathbb{R}^2$ as well:

First note that all local results of section \ref{sect:local}, in particular the Special Combined Angle Lemma \ref{thm:specialcombinedangle}, apply without alteration on surfaces of any curvature. The global results of section \ref{sect:global} on the other hand necessitate a closer look.

Note that the convex hull of a finite number of points in a plane of nonpositive curvature is still the geodesic polygon with vertices at some of these points. This can be shown using the fact that in a simply connected Riemannian manifold of nonpositive curvature, the exponential map at any point is a diffeomorphism (see theorem 2.6.6 in \cite{MR666697}). Therefore Lemma \ref{thm:convexhull} and \ref{thm:atleastthree} still apply (where, in the proofs, one replaces straight lines with geodesics). Also, the definition and results regarding the circumference still work without alteration. It is worth noting how we use that we are dealing with a metric on $\mathbb{R}^2$ and not an arbitrary surface: We used that the surface is simply connected and the definition of the circumference relies on the notion of an outer face, which uses that $\mathbb{R}^2\setminus G$ has exactly one unbounded component. Either of these arguments would fail on, say, the flat torus.

Our version of Gau\ss-Bonnet for essentially simple curves as given by Lemma \ref{thm:gaussbonnetessential} still applies with the following adjustment: Since we have nonpositive curvature, the turn angle along such a counterclockwise closed path now is $360\degrees$ \textit{or more}. Unlike in flat geometry, we won't have conditional path independence as described by Lemma \ref{thm:condpathind}. We will see below that we can do without that fact. Note that the First and Second Turn Lemma \ref{thm:firsturn}/\ref{thm:secondturn} again describe local properties and therefore still apply.

We will now prove a generalized version of the flat $60\degree$-Lemma \ref{thm:60degrees} before we prove the main result for nonpositive curvature. The proofs will closely follow the ideas of the flat case.

\begin{lemma}[$60\degrees$ Lemma, nonpositive curvature]\label{thm:60degreesnegative}
	Consider a path going through four edges $a,b,c,d$ of a geodesic net on $\mathbb{R}^2$ with a metric of nonpositive curvature and three vertices $u,v,w$ in the order $a,u,b,v,c,w,d$. Assume that
	\begin{enumerate*}
		\item $a\neq -b$, $c\neq -d$
		\item $u$ and $w$ are balanced ($v$ can be balanced or not).
		\item $b$ immediately follows $a$ at $u$ (i.e. $a*b$ takes the first right turn).
		\item $d$ immediately follows $c$ at $w$ (i.e. $c*d$ takes the first right turn).
		\item The convex hull of $u,v,w$ contains no unbalanced vertices (except, possibly, $v$ itself).
	\end{enumerate*}
	Then there exists a piecewise geodesic path $\gamma$ contained in the convex hull of $u$, $v$, $w$ with the following properties:
	\begin{itemize*}
		\item $\gamma$ starts at $u$ and ends at $w$.
		\item The turn angle along $a*\gamma*d$ is $60\degrees$ or less.
		\item $\gamma$ is simple apart from possible admissible backtracks.
	\end{itemize*}
\end{lemma}

Note that we do not require the path to be on the geodesic net. The important difference between the flat version and the version for nonpositive curvature is that we are merely looking for \emph{some} path along which the turn angle is $60\degrees$ or less. This is how we deal with the absence of conditional path independence.

\begin{proof}
	We can again exclude the two trivial cases where $u=w$ or where the turn angle from $b$ to $c$ is nonpositive. In those cases, the path $\gamma=b*c$ fulfils the required properties.
	
	We arrive at a picture similar to the one in figure \ref{fig:60degrees} with a geodesic triangle bounded by $b$, $c$ where $\ell$ is the unique geodesic going through $u$ and $w$. As before, $a$ or $d$ could also be below $\ell$ (for a consideration of what \enquote{below} means in this case, see Case 2).

	We define a path $\gamma$ as before: start at $u$. The first edge of $\gamma$ is given by the second right turn at $u$, counting from $a$. Then $\gamma$ always takes the first right turn (unless it leads to $v$, in which case it takes the second right turn) and terminates at $w$ or once we leave the convex hull. The argument that $\gamma$ is well-defined still applies. For clarification, in the argument that the path terminates, the \enquote{rays} from $v$ would now be the geodesic rays starting at $v$ and sweeping out the convex hull/geodesic triangle of $u$, $v$ and $w$.
	
	We again work on the following cases:
	
	\textbf{Case 1 $\gamma$ reaches $w$.} Note that by the same arguments as in the flat case, $\gamma$ fulfills all properties required of a path according to the lemma, except possibly the turn angle. If the turn angle along $a*\gamma*d$ is nonpositive, we are therefore done.
	
	Otherwise, at some point the turn angle at a vertex on $a*\gamma*d$ must be positive. Denote by $x$ the first edge with a positive turn angle and by $y$ the last edge with a positive turn angle. Follow the same construction as in the flat case and arrive at a path $a*\gamma(\rightarrow x)* e_v* f_v* \gamma(y\rightarrow)*d$. Note that now $\gamma':=\gamma(\rightarrow x)* e_v* f_v* \gamma(y\rightarrow)$  fulfills all requirements of the lemma, including the turn angle along $a*\gamma'*d$ which is, as argued before, at most $-60\degrees+180\degrees-60\degrees=60\degrees$. The argument that $\gamma'$ is simple apart from one possible backtrack is the same as in the flat case.

	\textbf{Case 2 $\gamma$ leaves the convex hull of $u,v,w$ and terminates.}
	
	Note that the terms of $a$ or $d$ lying \enquote{below $\ell$} or \enquote{above $\ell$} that we use in the following are to be understood in the sense that their tangent vectors at the vertex lie below or above the tangent vector of $\ell$ in the tangent space at the respective vertex, which is divided into two half planes by $\ell$ where the lower half plane $u$ is the one including the tangent vector of $b$ and the lower half plane at $w$ is the one including the tangent vector of $c$. Denote by $\overline{\ell}$ the segment of $\ell$ between $u$ and $w$.
	
	\textbf{Case 2a both $a$ and $d$ lie above $\ell$.}	We find our vertex $x$ on the path $\gamma$ as follows: If the turn angle from $a$ to the first edge of $\gamma$ is positive (i.e. the turn angle at $u$ is positive), then we choose $x:=u$. Otherwise we must enter the interior of the convex hull of $u$, $v$ and $w$ and can consider the following path: Start at $u$, follow $\gamma$ until it would leave the convex hull (at which point it crosses $\overline{\ell}$) and then return to $u$ along $\overline{\ell}$. This is a simple, closed, counterclockwise path. The total turn angle along this path must be $360\degrees$ or more (since it encloses a region of nonpositive curvature) and has three or more vertices. In particular, one of the vertices not incident to $\overline{\ell}$ must have a positive turn angle. Starting at $u$ going counterclockwise, the first such vertex will be denoted by $x$. The vertex $y$ can be found on a path starting with the second left turn at $d$ and then always taking the first left turn, except if it leads to $v$ (as before, this is the mirror image of the process of finding $\gamma$ an $x$) and we arrive at the familiar situation where we can define a path $u\to x\to v\to y\to w$ fulfilling the required properties.
	
	\textbf{Case 2b $d$ lies on or below $\ell$.} As before, this case is rather pathological.
	
	Consider the following path which we call $\gamma_+$: 	Truncate the path $\gamma$ so that it terminates at the point where it intersects $\overline{\ell}$ for the last time. This means it most likely won't terminate on a vertex but in the middle of an edge. Now continue along $\overline{\ell}$ to $w$.
	
	If the turn angle along $a*\gamma_+*d$ is $60\degrees$ or less, we are done (the argument for $\gamma_+$ being essentially simple is the same as previously). Otherwise note: The total turn angle along $a*\gamma_+*d$ must be more than $60\degrees$, the last turn of $\gamma+$ onto $\overline{\ell}$ must be nonpositive (since the last edge of $\gamma$ is on or below $\overline{\ell}$) and the turn from $\gamma_+$ onto $d$ is nonpositive (since $d$ lies on or below $\ell$). So there must be at least one vertex along $a*\gamma_+$ with a positive turn angle and this is in fact a turn between two edges of the geodesic net. Call the first such vertex $x$ and the last such vertex $y$. As usual, we allow $x$ and $y$ to coincide. Define $e_v$ and $f_v$ as before and we get a path $\gamma':=\gamma_+(\rightarrow x)* e_v* f_v* \gamma_+(y\rightarrow)$ such that the turn angle along $a*\gamma'*d$ can be at most $60\degrees$. Again the argument that $\gamma'$ is simple apart from one possible admissible backtrack remains the same. We now have the path we were looking for.
		
	\textbf{Case 2c $a$ lies on or below $\ell$}. This case is still the mirror image of Case 2c.
\end{proof}

We can now prove the general version of the main theorem \ref{thm:main}, i.e. the case of three boundary vertices and nonpositive curvature.

Again, we will closely follow the proof for the flat case, i.e. the proof of Lemma \ref{thm:threevertices}.

\begin{lemma}\label{thm:threeverticesnonpos}
	A geodesic net on $\mathbb{R}^2$ with a metric of nonpositive curvature with three unbalanced vertices has at most one balanced vertex.
\end{lemma}

\begin{proof}
	Call the three unbalanced vertices $x$, $y$ and $z$. The initial observations regarding the circumference still apply to the nonpositive curvature case, so we can reduce to the case where the circumference travels along vertices in the order
	\begin{align*}
		x,u_1,\dots,u_r,y,v_1,\dots,v_s,z,w_1,\dots,w_t,x
	\end{align*}
	Where all the $u_i,v_j,w_k$ are balanced vertices with $r,s,t\geq 1$. We will argue that $r=s=t=1$.
 	
 	Looking back at figure \ref{fig:labelthehex}, we have the setup of the $60\degree$-Lemma around the vertex $y$. We are using the non-flat version of the lemma this time and get a piecewise geodesic path inside the convex hull of $u_r$, $y$, $v_1$ starting at $u_r$ and ending at $v_1$ that is simple apart from admissible backtracks and such that the turn angle from the edge entering $u_r$ along this path to the edge leaving $v_1$ is at most $60\degrees$. Call this path $\alpha_y$. In a similar fashion, we get paths $\alpha_z$ and $\alpha_x$. Recall that the circumference is denoted by $\gamma$.
 	
 	Assume that $r>1$ which means we can define the following closed counterclockwise path:
 	\begin{align*}
 		\cdots\xrightarrow{\gamma} \underbrace{u_r\xrightarrow{\alpha_y} v_1}_{\leq 60\degrees}\xrightarrow{\gamma} v_s\xrightarrow{\gamma} \underbrace{z}_{\leq180\degrees}\xrightarrow{\gamma} w_1\xrightarrow{\gamma} \underbrace{w_t\xrightarrow{\alpha_x} u_1}_{\leq 60\degrees}\xrightarrow{\gamma}\cdots
 	\end{align*}
 	The angles given above are the turn angles along the respective parts of this path. It is important to point out that we assume $u_1\neq u_r$ which justifies that the two segments of $\leq60\degrees$ don't overlap.
 	
 	All other turn angles along this path are negative. Note that Gau\ss-Bonnet for essentially simple curves as stated in Lemma \ref{thm:gaussbonnetessential} applies: The argument that the circumference $\gamma$ is essentially simple was provided in the flat case. Both $\alpha_x$ and $\alpha_y$ are simple apart from admissible backtracks as given by Lemma \ref{thm:60degreesnegative}. Also, we are not using the parts of $\gamma$ visiting $x$ and $y$, therefore $\alpha_x$ and $\alpha_y$ can not meet $\gamma$ except at their endpoints. Therefore, the resulting path is still essentially simple. So the turn angle along this path should be $360\degrees$ or more. However, we can see above that this path has a total turn angle of $300\degrees$ or less, a contradiction.
 	
	$r=1$ follows. By analogous arguments $s=t=1$ and we arrive at the situation shown in figure \ref{fig:thehexnegative}.
	
	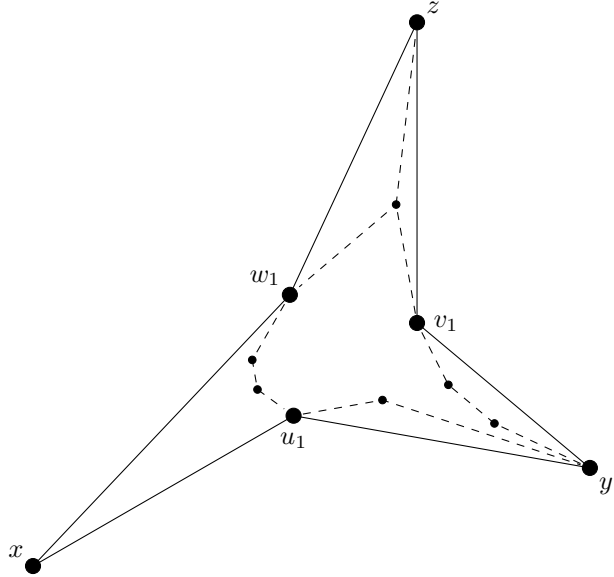
\begin{figure}
	\centering
	\begin{tikzpicture}[scale=2]
		\draw (0,0) node (x) {} -- ++(30:2) node (u) {} -- ++(-10:2) node (y) {} -- ++ (140:1.5) node (v) {}
		-- ++ (90:2) node (z) {} -- ++ (245:2) node (w) {} -- cycle;
		\draw [fill=black] (x) circle [radius=0.05];
		\draw [fill=black] (u) circle [radius=0.05];
		\draw [fill=black] (y) circle [radius=0.05];
		\draw [fill=black] (v) circle [radius=0.05];
		\draw [fill=black] (z) circle [radius=0.05];
		\draw [fill=black] (w) circle [radius=0.05];
		\node at (x) [above left] {$x$};
		\node at (u) [below=0.1cm] {$u_1$};
		\node at (y) [below right] {$y$};
		\node at (v) [right=0.1cm] {$v_1$};
		\node at (z) [above right] {$z$};
		\node at (w) [above left] {$w_1$};
		\draw [dashed] (u) -- ++(10:0.6) node (gy1) {} -- (y) -- ++(155:0.7) node (gy2) {} -- ++(140:0.4) node (gy3) {} -- (v);
		\draw [fill=black] (gy1) circle [radius=0.025];
		\draw [fill=black] (gy2) circle [radius=0.025];
		\draw [fill=black] (gy3) circle [radius=0.025];
		\draw [dashed] (v) -- ++(100:0.8) node (gz1) {} -- (z);
		\draw [dashed] (gz1) -- (w);
		\draw [fill=black] (gz1) circle [radius=0.025];
		\draw [dashed] (w) -- ++(240:0.5) node (gx1) {} -- ++(280:0.2) node (gx2) {} -- (u);
		\draw [fill=black] (gx1) circle [radius=0.025];
		\draw [fill=black] (gx2) circle [radius=0.025];
	\end{tikzpicture}
	\caption{The situation described in the final steps of the proof for nonpositive curvature. The solid path is the circumference $\gamma$, the three dashed paths are $\alpha_x$, $\alpha_y$ and $\alpha_z$ respectively. These paths could reach $x$, $y$ or $z$ as depicted. Note that since the situation where $u_1$, $v_1$ and $w_1$ don't coincide is shown to be inadmissible later in the proof, the angles in this picture can't be true to the angles as stated in the proof.}
	\label{fig:thehexnegative}
	\end{figure}
	
	We now consider the following closed path (for further reference, consider figure \ref{fig:thehexnegative}) which we call $\beta$:
	\begin{align*}
		\cdots\xrightarrow{\gamma} \underbrace{u_1\xrightarrow{\alpha_y} v_1}_{\leq 60\degrees}\xrightarrow{\gamma} \underbrace{z\vphantom{y}}_{\leq180\degrees}\xrightarrow{\gamma}  \underbrace{w_1\xrightarrow{\alpha_x} u_1}_{\leq 60\degrees}\xrightarrow{\gamma}\underbrace{y}_{\leq180\degrees}\xrightarrow{\gamma}  \underbrace{v_1\xrightarrow{\alpha_z} w_1}_{\leq 60\degrees}\xrightarrow{\gamma} \underbrace{x\vphantom{y}}_{\leq 180\degrees}\xrightarrow{\gamma}\cdots
	\end{align*}
	
	Note that all angles that are not specified are negative and therefore the sum of all turn angles along this path $\beta$ is $720\degrees$ or less. We will use the notation $\textrm{Turn}(\beta)\leq720\degrees$.
	
	\textbf{Claim:} $\textrm{Turn}(\beta)=720\degrees$. 
	
	$\beta$ crosses itself transversally by design which means it is not essentially simple and we can't apply Gau\ss-Bonnet directly. We proceed as follows:
	
	Consider the circumference $\gamma$ (given by the solid lines in figure \ref{fig:thehexnegative}) and the path $\alpha:=\alpha_x*\alpha_y*\alpha_z$ (given by the dashed lines in figure \ref{fig:thehexnegative}). $\gamma$ is essentially simple as established before. Each of $\alpha_x$, $\alpha_y$ and $\alpha_z$ is simple apart from admissible backtracks and they lie in separate convex hulls. Therefore, their concatenation $\alpha$ is also essentially simple. So Gau\ss-Bonnet applies to each of $\gamma$ and $\alpha$ and
	\begin{align*}
		\left.
		\begin{array}{l}
		\textrm{Turn}(\gamma)\geq360\degrees\\
		\textrm{Turn}(\alpha)\geq360\degrees
		\end{array}\right\rbrace\Rightarrow\textrm{Turn}(\gamma)+\textrm{Turn}(\alpha)\geq 720\degrees
	\end{align*}
	If we show $\textrm{Turn}(\beta)=\textrm{Turn}(\gamma)+\textrm{Turn}(\alpha)$, the claim follows.
	
	In fact, note that $\beta$ follows $\alpha$ and $\gamma$, only switching between them at $u_1$, $v_1$ and $w_1$. So we need to show that the total turn angle of $\beta$ at $u_1$ (which it visits twice) is the same as the sum of the turn angles of $\gamma$ and $\alpha$ at $u_1$. The observation for $v_1$ and $w_1$ will then be the same.
	
	Consider the situation at $u_1$ as depicted in Figure \ref{fig:anglechange} and observe:
	\begin{itemize*}
		\item The turn angle of $\alpha$ is $d$ (dashed-dashed).
		\item The turn angle of $\gamma$ is $a$ (solid-solid).
		\item The turn angle of $\beta$ for each of the two visits is $b$ (solid-dashed) and $c$ (dashed-solid) respectively.
	\end{itemize*}
	We use a negative sign if we go backwards along one of these angles and can see in the figure that $a-c+d-b=0$ and therefore $b+c=a+d$. So indeed, the sum of the turn angles of $\alpha$ and $\gamma$ at $u_1$ is the same as the sum of the turn angles of $\beta$ at $u_1$.
	
	Following the same argument at $v_1$ and $w_1$, we can conclude that in fact $\textrm{Turn}(\beta)=\textrm{Turn}(\gamma)+\textrm{Turn}(\alpha)$.
	
	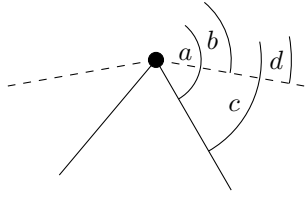
\begin{figure}
	\centering
	\begin{tikzpicture}[scale=2]
		\draw [fill=black] (0,0) circle [radius=0.05];
		\draw [dashed] (0,0) -- (-10:1);
		\draw [dashed] (0,0) -- (190:1);
		\draw (0,0) -- (-60:1);
		\draw (0,0) -- (230:1);
		\draw (50:0.3) arc [radius=0.3, start angle=50,end angle=-60];
		\draw (50:0.5) arc [radius=0.5, start angle=50,end angle=-10];
		\draw (10:0.7) arc [radius=0.7, start angle=10,end angle=-60];
		\draw (10:0.9) arc [radius=0.9, start angle=10,end angle=-10];
		\node at (10:0.2) {$a$};
		\node at (20:0.4) {$b$};
		\node at (-30:0.6) {$c$};
		\node at (0:0.8) {$d$};
	\end{tikzpicture}
	\caption{The situation at $u_1$. To simplify the picture, this sample situation is chosen so that all considered turn angles are negative.}
	\label{fig:anglechange}
	\end{figure}

	This finishes the proof of the claim. Note that $\textrm{Turn}(\beta)=720\degrees$ can only be achieved if all turn angles as specified in the definition of $\beta$ above are extremal. In particular, the angles at $x$, $y$ and $z$ must be equal to $180\degrees$. Therefore $u_1=v_1=w_1$ is the single balanced vertex of this geodesic net.
\end{proof}

\section{An example with four unbalanced vertices}
\label{sect:construct4}

In the following, we will give a detailed construction of the example geodesic net in flat $\mathbb{R}^2$ with four unbalanced vertices and $27$ balanced vertices given in Figure \ref{fig:fournet}. This is, for the moment, the maximal number of balanced vertices that the author could construct given just four unbalanced vertices.

We will write $ABC$ for the counterclockwise angle from $A$ to $C$ at $B$ and use the following classical theorem of Euclidean Geometry:
\begin{theorem}
	\label{thm:thales}
	Consider a circle with centre $O$ and three points $A$, $B$, $C$ on this circle, in clockwise order. Then the angle $AOC$ is twice as large as the angle $ABC$.
\end{theorem}
Using this theorem, we can construct the geodesic net. For a visualization, refer to Figure \ref{fig:construct4}.

\begin{figure}
	\centering
	
	\begin{tikzpicture}
	
	\node at (0,0) (Q) {};
	\node at (150:3) (X) {};
	\node at (125:3) (Y1) {};
	\node at (90:3) (Y2) {};
	\node at (55:3) (Y3) {};
	\node at (30:3) (Z) {};
	
	\node at (0,15) (P) {};
	\path (P) -- ++(210:6) node (A) {};
	\path (P) -- ++(255:6) node (B1) {};
	\path (P) -- ++(270:6) node (B2) {};
	\path (P) -- ++(285:6) node (B3) {};
	\path (P) -- ++(330:6) node (C) {};
	\path (P) -- ++(330:0.7) node (anglearcstart) {};
	
	\node at (-.84,4.48) (L) {};
	\node at (.84,4.48) (N) {};
	
	\draw [fill=black] (P) circle [radius=0.05] node [above=0.1] {$P$};
	\node at (P) [right=1cm] {angle $=240\degrees$};
	\draw [fill=black] (A) circle [radius=0.05] node [left=0.1] {$A$};
	\draw [fill=black] (C) circle [radius=0.05] node [right=0.1] {$C$};
	\draw [fill=black] (B1) circle [radius=0.05] node [below left=0.1] {$B_1$};
	\draw [fill=black] (B2) circle [radius=0.05] node [below right=0.1] {$B_2$};
	\draw [fill=black] (B3) circle [radius=0.05] node [below right=0.1] {$B_3$};
	
	\draw [fill=black] (Q) circle [radius=0.05] node [below=0.1] {$Q$};
	\node at (Q) [right=1cm] {angle $=240\degrees$};
	\draw [fill=black] (X) circle [radius=0.05] node [left=0.1] {$X$};
	\draw [fill=black] (Z) circle [radius=0.05] node [right=0.1] {$Z$};
	\draw [fill=black] (Y1) circle [radius=0.05] node [above left=0.1] {$Y_1$};
	\draw [fill=black] (Y2) circle [radius=0.05] node [above right=0.1] {$Y_2$};
	\draw [fill=black] (Y3) circle [radius=0.05] node [above right=0.1] {$Y_3$};

	\draw [dashed] (A.center) -- (P.center) -- (C.center);
	\draw [dashed] (A) arc [radius = 6,start angle = 210, end angle = 330];
	\draw [dashed] (150:0.7) arc [radius=0.7,start angle = 150, end angle = 390];

	\draw [dashed] (X.center) -- (Q.center) -- (Z.center);
	\draw [dashed] (Z) arc [radius = 3,start angle = 30, end angle = 150];
	\draw [dashed] (anglearcstart) arc [radius=0.7,start angle = -30, end angle = 210];	
	
	\draw (A.center) -- (B1.center) -- (C.center);
	\draw (A.center) -- (B2.center) -- (C.center);
	\draw (A.center) -- (B3.center) -- (C.center);
	
	\draw (X.center) -- (Y1.center) -- (Z.center);
	\draw (X.center) -- (Y2.center) -- (Z.center);
	\draw (X.center) -- (Y3.center) -- (Z.center);
	
	\draw (B1.center) -- (X.center);
	\draw (B3.center) -- (Z.center);
	
	\draw (Y1.center) -- (A.center);
	\draw (Y3.center) -- (C.center);
	
	\draw (B2.center) -- (Y2.center);
	
	\draw [dotted,line width=1] (A.center) -- (L.center);
	\draw [dotted,line width=1] (X.center) -- (L.center);
	\draw [dotted,line width=1] (Z.center) -- (N.center);
	\draw [dotted,line width=1] (C.center) -- (N.center);
	\draw [dotted,line width=1] (L.center) -- (N.center);
\end{tikzpicture}
	\caption{Construction of the geodesic net in Figure \ref{fig:fournet}. For better overview, the segments $XC$ and $ZA$ are not displayed. Using theorem \ref{thm:thales}, we know that if the angle at $P$ is $240\degrees$, then the angle $CB_iA$ will be $120\degrees$. The same is true for the angles $XY_iZ$. The subgraph $G'$ is given in dotted lines.}
	\label{fig:construct4}
\end{figure}
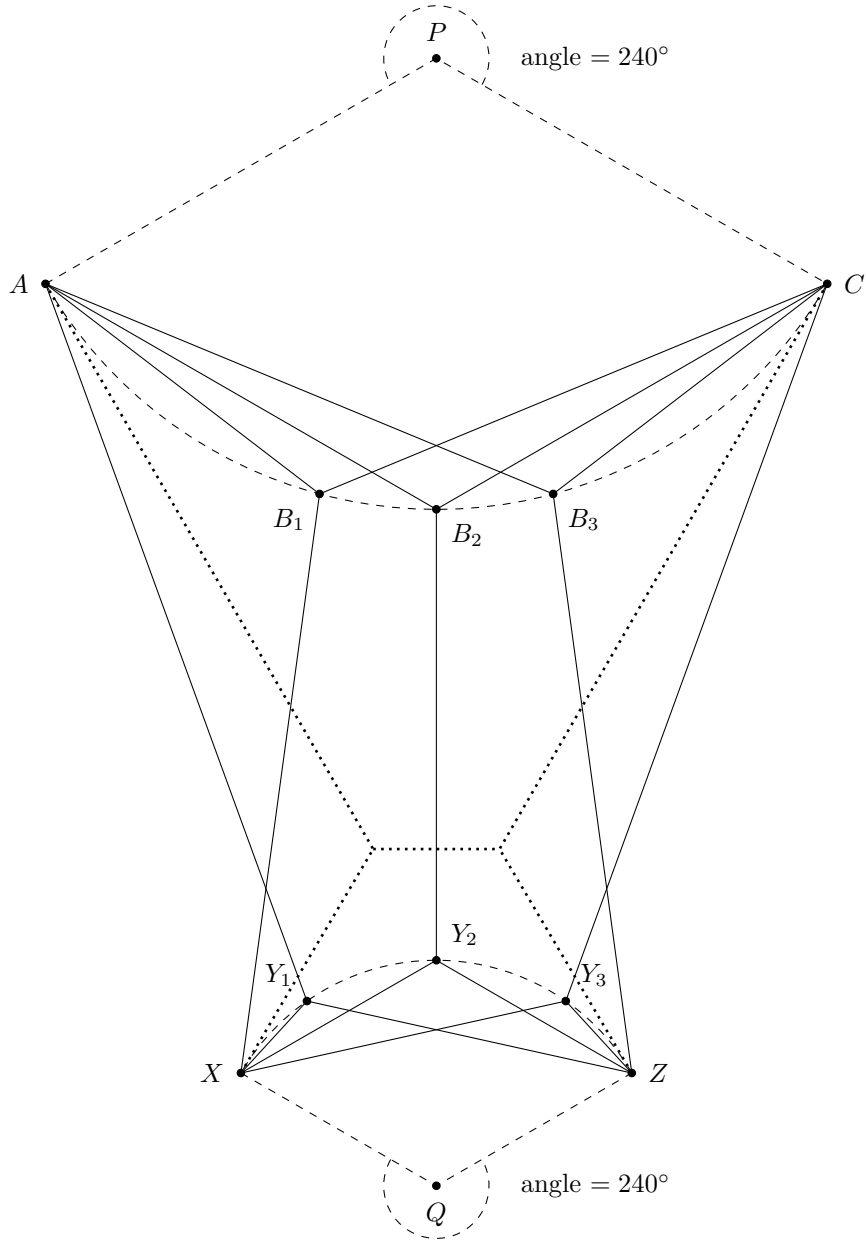

When describing circular arcs below we use the canonical angle on a circle where $0\degrees$ describes the rightmost point and we go counterclockwise. The choice of the units $1$, $2$ and $5$ during the construction will be justified below.

\begin{itemize*}
	\item Fix two points $P$ and $Q$ with distance $5$ units on a vertical axis. The line through $P$ and $Q$ will from now on be called \emph{the axis}.
	\item Draw a circular segment of radius $2$ centred at $P$ starting at angle $210\degrees$ and ending at angle $330\degrees$ (so that the outer angle of the circular segment will be $240\degrees$).
	\item Draw a circular segment of radius $1$ centred at $Q$ starting at angle $30\degrees$ and ending at angle $150\degrees$ (so that the outer angle of the circular segment will be $240\degrees$).
	\item The endpoints of these arcs are denoted by $A$, $C$ and $X$, $Z$ respectively. They will be the four unbalanced vertices of the geodesic net.
	\item Note that now, by Theorem \ref{thm:thales}, any point $B_i$ on the circular arc centred at $P$ will have the following property: $CB_iA=120\degrees$.
	\item By the same argument, any point $Y_i$ on the circular arc centred at $Q$ will give an angle $XY_iZ=120\degrees$.
	\item The vertices $B_2$ and $Y_2$ will be positioned at the intersection of the two circular arcs with the axis.
	\item Add the edges $AB_2$, $CB_2$, $B_2Y_2$, $XY_2$ and $ZY_2$. We now have two balanced degree $3$ vertices $B_2$ and $Y_2$.
	\item Find $B_1$ as follows: We want to add $B_1$ to the circular arc between $A$ and $B_2$. We will find the exact position as follows: Note that if $B_1=B_2$, then
		\begin{align*}
			XB_1C=XB_2C > Y_2B_2C=120\degrees
		\end{align*}
		and if $B_1=A$, then
		\begin{align*}
			XB_1C=XAC\leq 90\degrees
		\end{align*}
		Note that $XAC\leq 90\degrees$ is due to the fact that the distance from $X$ to the axis is smaller than the distance from $A$ to the axis. By continuity, there now must be some choice for $B_1$ on the circular arc between $A$ and $B_2$ such that $XB_1C=120\degrees$.
	\item Now add the edges $AB_1$, $CB_1$ and $XB_1$ making $B_1$ a balanced degree $3$ vertex.
	\item $Y_1$ can be found in a very similar fashion: Similar to above, we have to find the position of $Y_1$ on the arc between $X$ and $Y_2$. Note that if $Y_1=Y_2$, we would have
		\begin{align*}
			ZY_1A=ZY_2A>ZY_2B_2=120\degrees
		\end{align*}
		and if $Y_1=X$, then
		\begin{align*}
			ZY_1A=ZXA<120\degrees
		\end{align*}
	 	The inequality $ZXA<120\degrees$ will be justified below. Again by continuity there is a choice for $Y_1$ such that $ZY_1A=120\degrees$. Add the appropriate edges to make $Y_1$ a balanced degree $3$ vertex.
	\item $B_3$ and $Y_3$ are simply the reflection of $B_1$ and $Y_1$ along the axis. We add the appropriate edges to make them balanced degree $3$ vertices.
	\item Now add a subgraph $G'$ with two degree $3$ balanced vertices (shown with dotted lines in Figure \ref{fig:construct4}). That adding $G'$ is possible should be clear from the figure but will be justified further below.
	\item Finally, add the two segments $AZ$, and $CX$ (not depicted in Figure \ref{fig:construct4}) which will intersect at a vertex $M$ on the axis. This is the single degree $6$ vertex of the geodesic net.
\end{itemize*}

We now have a geodesic net with the unbalanced vertices $A$, $C$, $X$ and $Z$, with eight balanced degree $3$ vertices ($B_i$, $Y_i$ and the two vertices of $G'$), eighteen balanced degree $4$ vertices (these are all the intersections between straight line segments) and one balanced degree $6$ vertex ($M$).

During the construction, we made three seemingly liberal choices regarding length: The radii of the two arcs are given as $R=2$ and $r=1$ and the distance between $P$ and $Q$ was chosen to be $d=5$. While these are not the only possible choices, there are still some restrictions on these numbers:
\begin{itemize*}
	\item If $r=R$, the subgraph $G'$ added in the last step would intersect the axis in the point $M$. This would produce a geodesic net where $M$ is a vertex of degree $8$, however the total number of balanced vertices is larger if $M$ does not lie on $G'$. Therefore, $r\neq R$ gives more balanced vertices.
	\item Clearly, $d> R+r$ because otherwise the two arcs would intersect.
	\item We have to choose $R$, $r$, $d$ such that the angle $ZXA<120\degrees$, a fact we used above. It is an exercise in trigonometry that $R=2$, $r=1$ and $d=5$ is one such choice but of course not the only one. For the given numbers, we have $ZXA\approx 104\degrees$.
	\item Finally, we have to ensure that the subgraph $G'$ \enquote{fits} in the picture, i.e. that it is possible to position the two degree $3$ vertices in a way that allows all three angles at each vertex to be $120\degrees$. The choices we made for $R$, $r$, $d$ are one possibility for which, by basic arguments involving the angle sum in triangles and quadrilaterals, adding such a subgraph is possible.
\end{itemize*}
In fact, the restrictions on $R$, $r$ and $d$ allow for enough freedom that, with any two fixed, the third one can be slightly pertubed still allowing the same construction. In other words, the set of admissible $(R,r,d)$ is an open set in $\mathbb{R}^3$.

The geodesic net constructed in this section shows that in the context of Theorem \ref{thm:main}, we have $f(4)\geq 27$. This example only makes use of degree $3$, $4$ and $6$ balanced vertices and is also highly symmetric, suggesting that it might in fact not be a (or even \textit{the}) maximal example. In fact, it seems that once we go beyond three unbalanced vertices, we gain much more \enquote{freedom} as to how the balanced vertices can be distributed. In particular, for very large degree we should be able to construct with very small restrictions. In the context of the proof in this paper, the $60\degrees$ Lemma \ref{thm:60degrees} \enquote{looses its teeth}. These observations support the conjectures stated in section \ref{sect:intro}.

\section{Geodesic nets in the case of positive curvature}
\label{sect:positive}

Based on the result for nonpositive curvature, one should consider if similar statements can be made for geodesic nets on surfaces of positive curvature.

There is an example of a geodesic net on the round hemisphere with 3 unbalanced and 3 balanced vertices. The example is given in Figure \ref{fig:hemisphere} and can be constructed as follows:
\begin{itemize*}
	\item Parametrize the upper hemisphere by latitude ($0\degrees$ is the north pole, $90\degrees$ is the equator) and longitude (from $0\degrees$ to $360\degrees$).
	\item Position three vertices $A$, $B$ and $C$ as follows: they have longitude $0\degrees$, $120\degrees$ and $240\degrees$ respectively and all three have the same latitude $L$ which will be chosen below.
	\item Note that for each $0\degrees\leq L\leq 90\degrees$, the three vertices $A$, $B$ and $C$ form an equilateral geodesic triangle.
	\item At $L=0\degrees$, the triangle is empty, at $L=90\degrees$ the triangle has area $2\pi$. Therefore, for some $0\degrees<L<90\degrees$, the area is $\pi$. This is the latitude we choose for the three vertices.
	\item The result is a geodesic triangle where the interior angle at each of $A$, $B$ and $C$ is $120\degrees$. This follows directly from Gau\ss-Bonnet.
	\item Finally, position three vertices $X$, $Y$ and $Z$ at the equator at longitude $0\degrees$, $120\degrees$ and $240\degrees$ respectively.
	\item Add the edges $XA$, $YB$ and $ZC$.
\end{itemize*}
The result is a geodesic net with three unbalanced vertices $X$, $Y$, $Z$ and three balanced vertices $A$, $B$, $C$, each of degree three.

Note that we can avoid the equator altogether by moving up $X$, $Y$ and $Z$ to a latitude of less than $90\degrees$. This implies that the example can be constructed on a surface of positive curvature with no closed (nontrivial) geodesics.

It is obvious that we can extend this example to one on $\mathbb{R}^2$ with positive curvature. Actually, all that is needed for this example is an equiangular geodesic triangle with total curvature equal to $\pi$. Any surface with such a triangle allows this geodesic net to be embedded.

This observation raises the following question:

\emph{Is there an example of a geodesic net with three unbalanced vertices on $\mathbb{R}^2$ endowed with a metric of positive curvature such that the interior of the geodesic net has total curvature less than $\pi$ but still there is more than one balanced vertex?}

While we conjecture the answer to this question to be \enquote{No}, note that the proofs for the flat and negatively curved cases above make extensive use of the fact that the turn angle along any counterclockwise path is at least $360\degrees$, a fact that isn't true even for small amounts of positive curvature. This conjecture therefore necessitates a different proof.

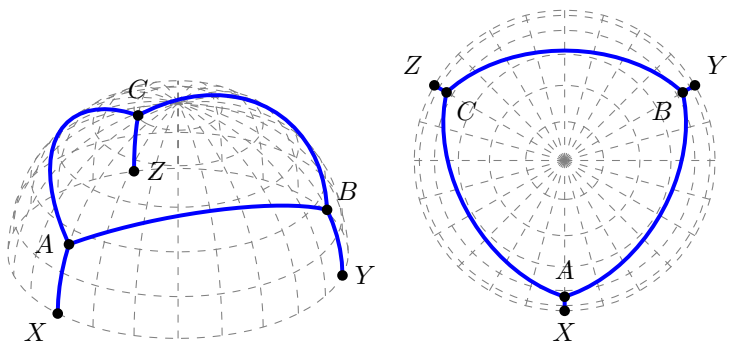
\begin{figure}
	\centering
	\tikzset{perspective/.style={scale=2,x={(-0.8cm,-0.4cm)},y={(0.8cm,-0.4cm)},
    z={(0cm,1cm)}},
    perspective2/.style={scale=2,x={(0cm,-1cm)},y={(1cm,0cm)},
    z={(0cm,0cm)}},    
    points/.style={fill=black,draw=black,thick}
    }
	\begin{tikzpicture}[perspective]

	\def\alpha{25}
	\def\rotangle{43}
	\def\edgefrom{40}
	\def\edgefromtwo{45}
	\def\edgeuntil{140}

	\coordinate (X) at (1,0,0);
	\coordinate (Y) at ({cos(120)},{sin(120)},0);
	\coordinate (Z) at ({cos(240)},{sin(240)},0);
	
	\coordinate (A) at ({cos(0)*cos(\alpha)},{sin(0)*cos(\alpha)},{sin(\alpha)});
	\coordinate (B) at ({cos(120)*cos(\alpha)},{sin(120)*cos(\alpha)},{sin(\alpha)});
	\coordinate (C) at ({cos(240)*cos(\alpha)},{sin(240)*cos(\alpha)},{sin(\alpha)});

	\foreach \t in {0,15,...,165}
			{\draw[gray,dashed] ({cos(\t)},{sin(\t)},0)
				\foreach \rho in {5,10,...,180}
					{--({cos(\t)*cos(\rho)},{sin(\t)*cos(\rho)},
          {sin(\rho)})};
			}
	\foreach \t in {0,15,...,75}
			{\draw[gray,dashed] ({cos(\t)},0,{sin(\t)})
				\foreach \rho in {5,10,...,360}
					{--({cos(\t)*cos(\rho)},{cos(\t)*sin(\rho)},
          {sin(\t)})}--cycle;
			}
	
	\draw[blue,line width=1.5] (X)
				\foreach \rho in {5,10,...,\alpha}
					{--({cos(0)*cos(\rho)},{sin(0)*cos(\rho)},
          {sin(\rho)})};	
    \draw[blue,line width=1.5] (Y)
				\foreach \rho in {5,10,...,\alpha}
					{--({cos(120)*cos(\rho)},{sin(120)*cos(\rho)},
          {sin(\rho)})};	
    \draw[blue,line width=1.5] (Z)
				\foreach \rho in {5,10,...,\alpha}
					{--({cos(240)*cos(\rho)},{sin(240)*cos(\rho)},
          {sin(\rho)})};	
	
	\draw[blue,line width=1.5] (A)
				\foreach \rho in {\edgefrom,\edgefromtwo,...,\edgeuntil} 
					{--({cos(-30)*cos(\rho)-sin(-30)*cos(\rotangle)*sin(\rho)},{sin(-30)*cos(\rho)+cos(-30)*cos(\rotangle)*sin(\rho)},
          {sin(\rotangle)*sin(\rho)})} -- (B);
          
	\draw[blue,line width=1.5] (B)
				\foreach \rho in {\edgefrom,\edgefromtwo,...,\edgeuntil} 
					{--({cos(-30+120)*cos(\rho)-sin(-30+120)*cos(\rotangle)*sin(\rho)},{sin(-30+120)*cos(\rho)+cos(-30+120)*cos(\rotangle)*sin(\rho)},
          {sin(\rotangle)*sin(\rho)})} -- (C);   
    
	\draw[blue,line width=1.5] (C)
				\foreach \rho in {\edgefrom,\edgefromtwo,...,\edgeuntil} 
					{--({cos(-30+240)*cos(\rho)-sin(-30+240)*cos(\rotangle)*sin(\rho)},{sin(-30+240)*cos(\rho)+cos(-30+240)*cos(\rotangle)*sin(\rho)},
          {sin(\rotangle)*sin(\rho)})} -- (A);

	\fill[points] (X) ellipse (0.8pt and 0.8pt) node [below left] {$X$};
	\fill[points] (Y) ellipse (0.8pt and 0.8pt) node [right=1pt] {$Y$};
	\fill[points] (Z) ellipse (0.8pt and 0.8pt) node [right=1pt] {$Z$};
	
	\fill[points] (A) ellipse (0.8pt and 0.8pt) node [left=2pt] {$A$};
	\fill[points] (B) ellipse (0.8pt and 0.8pt) node [above right] {$B$};
	\fill[points] (C) ellipse (0.8pt and 0.8pt) node [above=3pt] {$C$};
\end{tikzpicture}	
\begin{tikzpicture}[perspective2]

	\def\alpha{25}
	\def\rotangle{43}
	\def\edgefrom{40}
	\def\edgefromtwo{45}
	\def\edgeuntil{140}

	\coordinate (X) at (1,0,0);
	\coordinate (Y) at ({cos(120)},{sin(120)},0);
	\coordinate (Z) at ({cos(240)},{sin(240)},0);
	
	\coordinate (A) at ({cos(0)*cos(\alpha)},{sin(0)*cos(\alpha)},{sin(\alpha)});
	\coordinate (B) at ({cos(120)*cos(\alpha)},{sin(120)*cos(\alpha)},{sin(\alpha)});
	\coordinate (C) at ({cos(240)*cos(\alpha)},{sin(240)*cos(\alpha)},{sin(\alpha)});

	\foreach \t in {0,15,...,345}
			{\draw[gray,dashed] ({cos(\t)},{sin(\t)},0)
				\foreach \rho in {5,10,...,90}
					{--({cos(\t)*cos(\rho)},{sin(\t)*cos(\rho)},
          {sin(\rho)})};
			}
	\foreach \t in {0,15,...,75}
			{\draw[gray,dashed] ({cos(\t)},0,{sin(\t)})
				\foreach \rho in {5,10,...,360}
					{--({cos(\t)*cos(\rho)},{cos(\t)*sin(\rho)},
          {sin(\t)})}--cycle;
			}
	
	\draw[blue,line width=1.5] (X)
				\foreach \rho in {5,10,...,\alpha}
					{--({cos(0)*cos(\rho)},{sin(0)*cos(\rho)},
          {sin(\rho)})};	
    \draw[blue,line width=1.5] (Y)
				\foreach \rho in {5,10,...,\alpha}
					{--({cos(120)*cos(\rho)},{sin(120)*cos(\rho)},
          {sin(\rho)})};	
    \draw[blue,line width=1.5] (Z)
				\foreach \rho in {5,10,...,\alpha}
					{--({cos(240)*cos(\rho)},{sin(240)*cos(\rho)},
          {sin(\rho)})};	
	
	\draw[blue,line width=1.5] (A)
				\foreach \rho in {\edgefrom,\edgefromtwo,...,\edgeuntil} 
					{--({cos(-30)*cos(\rho)-sin(-30)*cos(\rotangle)*sin(\rho)},{sin(-30)*cos(\rho)+cos(-30)*cos(\rotangle)*sin(\rho)},
          {sin(\rotangle)*sin(\rho)})} -- (B);
          
	\draw[blue,line width=1.5] (B)
				\foreach \rho in {\edgefrom,\edgefromtwo,...,\edgeuntil} 
					{--({cos(-30+120)*cos(\rho)-sin(-30+120)*cos(\rotangle)*sin(\rho)},{sin(-30+120)*cos(\rho)+cos(-30+120)*cos(\rotangle)*sin(\rho)},
          {sin(\rotangle)*sin(\rho)})} -- (C);   
    
	\draw[blue,line width=1.5] (C)
				\foreach \rho in {\edgefrom,\edgefromtwo,...,\edgeuntil} 
					{--({cos(-30+240)*cos(\rho)-sin(-30+240)*cos(\rotangle)*sin(\rho)},{sin(-30+240)*cos(\rho)+cos(-30+240)*cos(\rotangle)*sin(\rho)},
          {sin(\rotangle)*sin(\rho)})} -- (A); 
	
	\fill[points] (X) ellipse (0.8pt and 0.8pt) node [below=1pt] {$X$};
	\fill[points] (Y) ellipse (0.8pt and 0.8pt) node [above right=1pt] {$Y$};
	\fill[points] (Z) ellipse (0.8pt and 0.8pt) node [above left=1pt] {$Z$};
	
	\fill[points] (A) ellipse (0.8pt and 0.8pt) node [above=3pt] {$A$};
	\fill[points] (B) ellipse (0.8pt and 0.8pt) node [below left] {$B$};
	\fill[points] (C) ellipse (0.8pt and 0.8pt) node [below right] {$C$};
\end{tikzpicture}
	\caption{A geodesic net on the unit hemisphere with three unbalanced vertices and three balanced vertices.}
	\label{fig:hemisphere}
\end{figure}

\begin{minipage}{\textwidth}
\section{Acknowledgements}

The author would like to thank his PhD advisor Alexander Nabutovsky for suggesting conjectures \ref{conj:limitless} and \ref{conj:imbalance}, numerous discussions and for carefully reading the first draft of this manuscript. He would also like to thank the two referees for helpful suggestions. This research was partially supported by an NSERC Vanier Scholarship.
\end{minipage}

\bibliographystyle{amsalpha}

\bibliography{../../bib/main}

\end{document}